\documentclass[a4paper,10pt]{article}
\usepackage{amsmath}
\usepackage{amsthm}
\let\newemptytheorem\newtheorem
\usepackage{amssymb}
\usepackage{bm}
\usepackage[margin=0.5in]{caption}
\usepackage{dsfont}
\usepackage{enumitem}
\usepackage{graphicx}
\usepackage{indentfirst}
\usepackage{mathrsfs}
\usepackage{mathtools}
\usepackage{nameref}
\usepackage[new]{old-arrows}
\usepackage{stmaryrd}
\usepackage{svg}
\usepackage[x11names]{xcolor}
\usepackage{xparse}
\usepackage{mathabx} 
\usepackage{blindtext}
\usepackage{titlesec}
\usepackage{titletoc}
\usepackage{mdframed}
\usepackage{letltxmacro}
\LetLtxMacro\amsproof\proof
\LetLtxMacro\amsendproof\endproof
\usepackage{thmbox}

\AtBeginDocument{
  \LetLtxMacro\proof\amsproof
  \LetLtxMacro\endproof\amsendproof
}

\newcommand{\deffont}[1]{\textbf{#1}}

\usepackage{sansmathfonts}
\usepackage[sfdefault]{atkinson}
\usepackage[T1]{fontenc}
\usepackage[colorlinks=true,
    citecolor=DeepSkyBlue4,
    linkcolor=black,
    urlcolor=DeepSkyBlue4,
    linktoc=page,
    hyperindex=true,
    pdfcreator={}]{hyperref} 
    
    \hypersetup{
        pdfauthor={Adrien Abgrall and Alexandra Gurieva and Camille Horbez},
        pdftitle={Algebraic and measurable embeddings between square-free right-angled Artin groups},
        pdfkeywords={Right-angled Artin groups, Embeddings, Measure equivalence, Extension graphs},
}

\usepackage{geometry}
\usepackage{setspace} %Interline space
\usepackage{enumitem}
\setlist{nosep}

\usepackage{graphics}
\normalsize
\usepackage{tcolorbox}

\newcounter{theo}[section]
\renewcommand{\thetheo}{\arabic{section}.\arabic{theo}}
\newenvironment{theo}[2][]{
\refstepcounter{theo}
\ifstrempty{#1}
{\mdfsetup{
frametitle={
\tikz[baseline=(current bounding box.east),outer sep=0pt]
\node[anchor=east,rectangle,fill=DeepSkyBlue4!15]
{\strut Theorem~\thetheo};}}
}
{\mdfsetup{
frametitle={
\tikz[baseline=(current bounding box.east),outer sep=0pt]
\node[anchor=east,rectangle,fill=DeepSkyBlue4!15]
{\strut Theorem~\thetheo:~#1};}}
}
\mdfsetup{innertopmargin=2pt,linecolor=black,
linewidth=1pt,topline=true,
frametitleaboveskip=\dimexpr-\ht\strutbox\relax
}
\begin{mdframed}[]\relax
\label{#2}}{\end{mdframed}}

\global\mdfdefinestyle{exampledefault}{topline=false,bottomline=false,rightline=false}

\theoremstyle{definition}
\newtheorem[S,bodystyle=\normalfont\noindent,titlestyle=\normalfont{~(#1)}]{thm}[theo]{Theorem}
\newtheorem[S,bodystyle=\normalfont\noindent,titlestyle=\normalfont{~(#1)}]{lem}[theo]{Lemma}
\newtheorem[S,bodystyle=\normalfont\noindent,titlestyle=\normalfont{~(#1)}]{coro}[theo]{Corollary}
\newtheorem[S,bodystyle=\normalfont\noindent,titlestyle=\normalfont{~(#1)}]{cor}[theo]{Corollary}
\newtheorem[S,bodystyle=\normalfont\noindent,titlestyle=\normalfont{~(#1)}]{prop}[theo]{Proposition}
\newtheorem[S,bodystyle=\normalfont\noindent,titlestyle=\normalfont{~(#1)}]{facts}[theo]{Facts}

\theoremstyle{remark}
\newemptytheorem{rem}[theo]{Remark}

\newemptytheorem{claim}[theo]{Claim}
\newemptytheorem{exer}[theo]{Exercise}
\newemptytheorem{q}[theo]{Question}
\theoremstyle{definition}
\newtheorem[S,bodystyle=\normalfont\noindent,titlestyle=\normalfont{~(#1)}]{defn}[theo]{Definition}
\newemptytheorem{ex}[theo]{Example}

\newtheorem[S,bodystyle=\normalfont\noindent,titlestyle=\normalfont{~(###1)}]{manualtheoreminner}{Theorem}
\newenvironment{manualtheorem}[1]{%
  \renewcommand\themanualtheoreminner{#1}%
  \manualtheoreminner
}{\endmanualtheoreminner}

\newcommand*{\subproofname}{Proof}
\newenvironment{subproof}[1][\subproofname]{\begin{proof}[#1]}{\end{proof}}

\newcommand{\gen}[1]{\left\langle {#1} \right\rangle}
\newcommand{\ngen}[1]{\left\langle \!\langle {#1}\right \rangle\!\rangle}

\newcommand{\NN}{\mathbb{N}}
\newcommand{\ZZ}{\mathbb{Z}}

\newcommand{\cat}{\mathrm{CAT}(0)}

\DeclareMathOperator{\lk}{lk}
\DeclareMathOperator{\st}{st}
\DeclareMathOperator{\clr}{clr}
\DeclareMathOperator{\supp}{supp}
\DeclareMathOperator{\xtype}{xtype}

\renewcommand{\leq}{\leqslant}
\renewcommand{\nleq}{\nleqslant}
\renewcommand{\geq}{\geqslant}

\renewcommand{\emptyset}{\varnothing}
\renewcommand{\epsilon}{\varepsilon}
\newcommand{\ext}{{}^{\mathrm{ext}}}

\newcommand{\mcl}{\mathcal}

\newcommand{\G}{\Gamma}

\newcommand{\w}{\omega}

\newcommand{\mclH}{\mcl{H}}
\newcommand{\mclG}{\mcl{G}}
\newcommand{\mclA}{\mcl{A}}

\newcommand{\rhoL}{\rho_{\Delta}}
\newcommand{\rhoG}{\rho_{\G}}

\newcommand{\AG}{A(\G)}

\newcounter{comments}

\title{Algebraic and measurable embeddings between square-free right-angled Artin groups}
\author{Adrien Abgrall \and Alexandra Gurieva \and Camille Horbez}
\date{\today}

\begin{document}

\maketitle

\begin{abstract}
    Let $\Gamma,\Delta$ be finite simplicial graphs, and assume that $\Gamma$ has no induced squares. We prove that the right-angled Artin group $A(\Delta)$ measurably embeds into $A(\Gamma)$ if and only if $A(\Delta)$ embeds as a subgroup in a graph product of free abelian groups over $\Gamma$. This in turn has a graph-theoretical characterisation which can be checked algorithmically.
    
    If additionally $A(\Gamma)$ and $A(\Delta)$ have cohomological dimension equal to two and $A(\Gamma)$ is not isomorphic to $\mathbb{Z}\times F_n$, we get that $A(\Delta)$ measurably embeds into $A(\Gamma)$ if and only if it embeds as a subgroup in $A(\Gamma)$.
    
    Our proof relies on the following algebraic statement. Let $K_1,\dots,K_n$ be a family of pairwise disjoint cliques in the square-free graph $\Gamma$, 
    %(or more generally in its extension graph $\Gamma\ext$), 
    and let $g_1,\dots,g_n$ be elements of $A(\Gamma)$ with respective parabolic supports $A(K_1),\dots,A(K_n)$. Then the subgroup of $A(\Gamma)$ generated by $g_1,\dots,g_n$ is a right-angled Artin group, where the only relations impose that $g_i$ and $g_j$ commute if $K_i\cup K_j$ is contained in a clique of $\Gamma$. % (or of $\Gamma\ext$).  
    The same holds true more generally for pairwise disjoint cliques $K_1,\dots,K_n$ in the extension graph $\Gamma\ext$, up to possibly replacing the elements $g_i$ by powers.  
\end{abstract}

\vspace{1.3cm}
\setcounter{tocdepth}{1}
\startcontents \printcontents{}{0}[1]{}

\vfill
\noindent \textcolor{black!70}{This work is openly licensed via \href{https://creativecommons.org/licenses/by/4.0/}{Creative Commons CC-BY 4.0}}
\smallskip

\noindent \textcolor{black!70}{This work was completed without use of artificial intelligence. The authors do not consent for all or part of this work to be processed by artificial intelligence models.}
\newpage

\section{Introduction}

Right-angled Artin groups (RAAGs) form a rich and diverse family of groups. First introduced by Baudisch \cite{Bau}, they now provide emblematic examples of groups with nice actions on $\mathrm{CAT}(0)$ cube complexes. Given a graph $\Gamma$ (always assumed to be simplicial, i.e.\ unoriented, with no loop-edge and no multi-edge), the \deffont{right-angled Artin group} $A(\Gamma)$ is the group having one generator per vertex of $\Gamma$, whose only relations are given by the commutation of any two generators corresponding to adjacent vertices. 
\medskip

More generally, given a graph $\Gamma$ and a family $(G_v)_{v\in V(\Gamma)}$ of groups indexed by the vertex set $V(\Gamma)$, the \deffont{graph product} of $(G_v)$ over $\Gamma$ was defined by Green \cite{green} to be the following group:
\[\bigast_{v\in V\Gamma}G_v \,/ \ngen{ [g_v,g_w] \mid \{v,w\}\in E(\Gamma),\,  g_v\in G_v,\, g_w\in G_w},\] where $E(\Gamma)$ denotes the edge set of $\Gamma$.
Right-angled Artin groups are exactly the groups that split as graph products with all vertex groups infinite cyclic.

The goal of this paper is to study embeddings between right-angled Artin groups, both from the algebraic viewpoint and the one of measured group theory.

\bigskip

Let $\Gamma$ and $\Delta$ be finite graphs. A theorem of Droms asserts that the RAAGs $A(\Gamma)$ and $A(\Delta)$ are isomorphic if and only if $\Gamma$ and $\Delta$ are isomorphic graphs \cite{Droms_isomorphism}. The problem of understanding when $A(\Delta)$ embeds as a subgroup of $A(\Gamma)$, in terms of a condition on the graphs $\Gamma,\Delta$, is still open in general. Significant progress in this direction was made by Kim--Koberda \cite{Kim_Koberda}, who solved the embedding problem in several situations, including when $A(\Gamma)$ is at most two-dimensional, i.e.\ $\Gamma$ does not contain any  triangle. See also \cite{Gen} for another viewpoint on the question.
\medskip

A crucial tool in their solution was the introduction of the \deffont{extension graph} $\Gamma\ext$ of $\Gamma$, a graph acted upon by $A(\Gamma)$ which is now often considered as an analogue for RAAGs of the curve graph of a surface, see \cite{Kim_Koberda_curve_graph}. It is defined as follows: its vertices correspond to conjugates of the cyclic subgroups associated to the standard generators of $A(\Gamma)$ (i.e.~the vertices of $\Gamma$), and two such conjugates are adjacent if and only if they commute. When $\Gamma$ is triangle free, Kim and Koberda prove that $A(\Delta)$ embeds in $A(\Gamma)$ if and only if $\Delta$ embeds as an induced subgraph of $\Gamma\ext$, and this condition was later proved to be algorithmically verifiable by Casals-Ruiz \cite{Casals-Ruiz}. Here we recall that $\Delta$ is \deffont{induced} in $\Gamma\ext$, which we denote by $\Delta\leq\Gamma\ext$, if two vertices of $\Delta$ are adjacent in $\Delta$ if and only if they are adjacent in $\Gamma\ext$. In general, without the triangle-freeness condition, Kim and Koberda prove that $\Delta\leq\Gamma\ext$ implies that $A(\Delta)\leq A(\Gamma)$, and conversely $A(\Delta)\leq A(\Gamma)$ implies that $\Delta$ embeds as an induced subgraph of the \deffont{clique graph} $\Gamma\ext_k$ of the extension graph. This is the graph whose vertices are the non-empty cliques of $\Gamma\ext$, where two cliques are joined by an edge if they are contained in a common clique. There are however examples showing that none of these conditions is an equivalence in general \cite{CRDK}, and the general picture is not yet understood.
\medskip

From the perspective of large-scale geometry, a lot of work has revolved around the classification of finitely generated right-angled Artin groups up to quasi-isometry, including \cite{BN,BJN,BKS,Hua1,Hua2,Mar,Oh,Oh-cor}. Much less is known about quasi-isometric embeddings from $A(\Delta)$ to $A(\Gamma)$. Recent work of Bader--Bensaid--Petyt has initiated this study, exhibiting both rigidity and flexibility phenomena for possible embeddings \cite{BBP2,BBP1}.  
\bigskip

Switching the perspective from geometric to measured group theory, the classification of finitely generated right-angled Artin groups up to \deffont{measure equivalence}, a measurable analogue of quasi-isometry introduced by Gromov in \cite{Gro}, has been investigated in \cite{Horbez_Huang,Escalier_Horbez,Horbez_Huang_2}. Known cases often require an assumption about the finiteness of the outer automorphism groups of the RAAGs at stake. In the present paper, we consider measurable embeddings, an order-like generalisation of measure equivalence that first appeared in work of Sako \cite[Definition~2.1]{Sak} and was also studied in \cite{DKLMT} under the name of measure subgroups, and in \cite{gurieva} in the context of mapping class groups. Given two finite graphs $\Gamma,\Delta$ that do not contain any induced square, we determine when $A(\Delta)$ measurably embeds into $A(\Gamma)$.
\medskip

The definition is as follows. Given two countable groups $G,H$, one says that $H$ \deffont{measurably embeds} into $G$ if there exists a standard measure space $\Sigma$ equipped with a measure-preserving action of $G\times H$ by Borel automorphisms, such that 
\begin{itemize}
    \item the $G$-action on $\Sigma$ is free and admits a measurable fundamental domain of finite positive measure;
    \item the $H$-action on $\Sigma$ is free and admits a measurable fundamental domain. 
\end{itemize}
Measure equivalence corresponds to the case where the fundamental domain for the $H$-action is also assumed to have finite measure. If $H$ is a subgroup of a countable group $G$, then $H$ measurably embeds in $G$, and measurable embedding is a transitive relation (see Remarks~\ref{transitivité du plongement ME} and \ref{alg embedding is meas embedding}).
\medskip

When the actions of $G$ and $H$ on $\Sigma$ have a common fundamental domain (of finite measure), the groups $G$ and $H$ are said to be \deffont{orbit equivalent}. The terminology comes from an observation of Furman \cite{Furman_oe} and Gaboriau \cite[Section~2.1]{gaboriau_examples}, showing that $G$ and $H$ are orbit equivalent if and only if they have free, measure-preserving actions on a standard probability space with the same orbits. In the same way as measurable embeddings are an asymmetric analogue of measure equivalence, an asymmetric version of orbit equivalence is the situation of a standard probability space $X$ equipped with two measure-preserving actions of $G,H$ such that for almost every $x\in X$, the orbit $H\cdot x$ in contained in $G\cdot x$. A celebrated theorem of Gaboriau and Lyons shows the existence of such an $X$ in the case where $H$ is the non-abelian free group $F_2$ and $G$ is any non-amenable countable group \cite{Gaboriau_Lyons}. Like orbit equivalence implies measure equivalence, the existence of such an $X$ implies that $H$ measurably embeds into $G$ (see e.g. \cite[Remark~2.36]{DKLMT}).

\bigskip
Before stating our main result, it is worth recalling a theorem of Dye ensuring that all finitely generated free abelian groups are orbit equivalent \cite{Dye} -- this was later famously generalised to all countably infinite amenable groups by Ornstein and Weiss \cite{OW}. This theorem, combined with an argument due to Gaboriau for free products \cite{gaboriau_examples} and extended to graph products in \cite{Horbez_Huang}, implies that $A(\Gamma)$ is always orbit equivalent to any graph product of countably infinite amenable groups (in particular, non-trivial finite-rank free abelian groups) over $\Gamma$. 

\medskip
Note that if $H$ is a graph product of non-trivial free abelian groups over $\Gamma$, then $H$ is itself a RAAG, over a graph $\widetilde\Gamma$ obtained from $\Gamma$ by blowing-up each vertex into a clique: the graph $\widetilde\Gamma$ comes with a \deffont{clique map}, i.e.\ a surjective map $f$ onto $\Gamma$, sending vertices to vertices and edges to either vertices or edges, and such that $f(v)=f(w)$ only if $v$ and $w$ have the same star. Following \cite{Horbez_Huang_2} and \cite[Section~9]{Escalier_Horbez}, we say that a graph $\Gamma$ is \deffont{clique reduced} if distinct vertices of $\Gamma$ have distinct stars, and observe that any graph $\Gamma$ has a clique map to a clique-reduced graph, which is unique up to isomorphism and called the \deffont{clique reduction} $\clr(\Gamma)$. The discussion from the previous paragraph ensures that $A(\Gamma)$ and $A(\clr(\Gamma))$ are always orbit equivalent.
\newpage

\begin{theo}[Measurable embeddings between square-free RAAGs]{thm:main}
Let $\Delta,\Gamma$ be finite graphs, and assume that $\Gamma$ has no induced squares. Then the following are equivalent:
\begin{enumerate}
    \item\label{main-1} $A(\Delta)$ measurably embeds into $A(\Gamma)$;
    \item\label{main-1,5} there exist free measure-preserving actions of $A(\Delta)$ and $A(\Gamma)$ on a standard probability space $X$ such that $A(\Delta)\cdot x\subseteq A(\Gamma)\cdot x$ for almost every $x\in X$; 
    \item\label{main-3} $A(\Delta)$ embeds as a subgroup in a graph product of  free abelian groups over $\Gamma$;
    %\item\label{main-2} $A(\Delta)$ embeds as a subgroup in the graph product of $\mathbb{Z}^{|V(\Delta)|}$ over $\Gamma$;
    \item\label{main-4} $\clr(\Delta)$ embeds as an induced  subgraph in $\Gamma\ext_k$.
\end{enumerate}
Additionally, there exists an algorithm which, given two finite graphs $\Delta,\Gamma$ as above, decides whether or not any of the above equivalent conditions holds. 
\end{theo}

As a corollary, if $\Delta,\Gamma$ are finite graphs, $A(\Delta)$ measurably embeds into $A(\Gamma)$, and $A(\Gamma)$ is coherent (i.e.\ its finitely generated subgroups are finitely presented), then $A(\Delta)$ is coherent (Corollary~\ref{cor:coherent}). We also prove that there is no \emph{universal finitely generated square-free RAAG for measurable embeddings}, i.e.~a finitely generated square-free RAAG in which every finitely generated square-free RAAG would measurably embed (Corollary~\ref{cor:no_universal}). Lastly, if $F_2\times F_2$ measurably embeds into $A(\Gamma)$, then it embeds as a subgroup (Remark~\ref{rem:square}).
\medskip

We also prove that (\ref{main-3}) can be made more quantitative with the following result.

\begin{theo}[Quantitative embedding in graph products of $\ZZ^n$]{thm:small_clique_expansion}
Let $\Gamma$ be a graph and let $\Delta$ be a finite graph. Assume that $A(\Delta)$ embeds in a graph product of free abelian groups over $\Gamma$. Then, $A(\Delta)$ embeds in a graph product of free abelian groups of rank $|V(\Delta)|$ over $\Gamma$.
\end{theo}

Note that $|V(\Delta)|$ is optimal for this result, as shown by the example where $\Delta$ is an $n$-clique and $\Gamma$ is a single vertex.

\smallskip
Intriguingly, in the setting of Theorem~\ref{thm:main}, characterising when $A(\Delta)$ algebraically embeds in $A(\Gamma)$, and knowing whether this can be algorithmically verifiable from their defining graphs, remains an open question. Likewise, the measure equivalence classification of RAAGs over square-free graphs is an open problem.

\begin{rem}
Conditions~(\ref{main-3}) and %(\ref{main-2})
(\ref{main-4}) are equivalent without the assumption that $\Gamma$ is finite. The implications (\ref{main-3})$\Rightarrow$(\ref{main-1,5})$\Rightarrow$(\ref{main-1}) also hold provided $\Gamma$ is countable. See Remark~\ref{rem:finite_unnecessary}. 

The assumption that $\Gamma$ has no induced squares cannot be dropped from the statement. We provide a counterexample to (\ref{main-4})$\Rightarrow$(\ref{main-3}) in Example~\ref{ex:square}, with $\Gamma$ a square, and $\Delta$ a line on four vertices. However, we do not know in this example whether $A(\Delta)$ measurably embeds into $A(\Gamma)=F_2\times F_2$. In general, it seems hard to determine which RAAGs measurably embed into $F_2\times F_2$. In contrast, a theorem of Rull states that any RAAG on a finite bipartite graph quasi-isometrically embeds into $F_2\times F_2$. As another hint of the subtleness of understanding measurable embeddings into $F_2\times F_2$, the following question is open.
\end{rem}

\begin{q}
Does there exist a standard probability space admitting two free, measure-preserving actions, one of $F_2=\langle a,b\rangle$, the other of $F_2\times F_2=A\times B$, so that for almost every $x\in X$, one has $\langle a\rangle\cdot x\subseteq A\cdot x$ and $\langle b\rangle\cdot x\subseteq B\cdot x$?
\end{q}

\begin{rem}\label{rem:meas-not-alg}
Even if $\Delta$ is clique reduced, $A(\Delta)$ can embed measurably into $A(\Gamma)$ without embedding as a subgroup. For instance, if $\Gamma$ is a pentagon, and $\Delta$ is  obtained from $\Gamma$ by adding one extra vertex joined to exactly two adjacent vertices of $\Gamma$, then $A(\Delta)$ measurably embeds into $A(\Gamma)$ and all equivalent conditions from Theorem~\ref{thm:main} are satisfied. Indeed, let $v_1,\dots,v_5$ be the vertices of the pentagon in $\Delta$, and $\hat v$ be its extra vertex (say it is adjacent to $v_1$ and $v_2$); let $w_1,\dots,w_5$ be the vertices of $\Gamma$. Then $\Delta$ embeds as an induced subgraph of $\Gamma\ext_k$, by sending each $v_i$ to $w_i$, and $\hat v$ to the edge $[w_1,w_2]$, proving that  (\ref{main-4}) holds. Algebraically, $A(\Delta)$ embeds as a subgroup in a graph of product of $\mathbb{Z}^2$ over $\Gamma$ (proving that (\ref{main-3}) holds) as follows: letting $\{a_i,b_i\}$ be a basis of the $\mathbb{Z}^2$ vertex group associated to $w_i$, one sends every $v_i$ to $a_i$, and $\hat v$ to the product $b_1b_2$. Note however that $A(\Delta)$ does not embed as a subgroup of $A(\Gamma)$ because $A(\Delta)$ contains $\mathbb{Z}^3$ while $A(\Gamma)$ does not.  
\end{rem}

In contrast with the previous remark, there are situations where the existence of a measurable embedding coincides with the existence of an algebraic embedding. For instance, say that a right-angled Artin group $A(\Gamma)$ is \deffont{two-dimensional} if its cohomological dimension is equal to two -- equivalently
$\Gamma$ does not contain any induced triangle and $A(\Gamma)$ is not a free group. We obtain the following statement, which is derived from a more technical version given in Theorem~\ref{thm:meas-algebraic}. 

\begin{theo}[The two-dimensional case]{thm:intro-meas-algebraic}
Let $\Gamma,\Delta$ be finite graphs such that $A(\Gamma),A(\Delta)$ are two-dimensional and $\Gamma$ does not contain any induced square.
\begin{itemize}
    \item If $\Gamma$ is a star graph, then $A(\Delta)$ mesurably embeds into $A(\Gamma)$ if and only if $\clr(\Delta)$ is an edgeless graph or a star graph.
    \item If $\Gamma$ is not a star graph, then $A(\Delta)$ measurably embeds into $A(\Gamma)$ if and only if $A(\Delta)$ embeds as a subgroup in $A(\Gamma)$.
\end{itemize}
\end{theo}

The case where $\Gamma$ is a star graph needs to be separate as the two-dimensional $\ZZ^2 \ast \ZZ^2$ measurably embeds into $\ZZ\times F_n$ for all $n\geq 2$ (because $\mathbb{Z}^2\ast\mathbb{Z}^2$ is orbit equivalent to $F_2$), yet $\ZZ^2 \ast \ZZ^2$ does not algebraically embed in $\ZZ\times F_n$ for any $n$ (the image of any $\ZZ^2$ would contain a central element). In fact, that case follows directly from Theorem~\ref{thm:main} and does not even require $A(\Delta)$ to be two-dimensional.

\subsection*{A word on the proof of Theorem~\ref{thm:main}.} The implication (\ref{main-3})$\Rightarrow$(\ref{main-1,5}) follows from the fact that any graph product of countably infinite free abelian groups is orbit equivalent to $A(\Gamma)$, after noting that in (\ref{main-3}) one can always assume the vertex groups to be countable and infinite. Moreover, (\ref{main-1,5})$\Rightarrow$(\ref{main-1}) is a general fact about measurable embeddings, see \cite[Remark~2.36]{DKLMT}. We now say a word about the
two main implications. Our proof of Theorem~\ref{thm:intro-meas-algebraic} is a simpler variant of the same argument; we will touch upon this in the final paragraph of our proof sketch of the implication (\ref{main-1})$\Rightarrow$(\ref{main-4}). The algorithmicity of our conditions relies on work of Casals-Ruiz \cite{Casals-Ruiz} and is explained in  Section~\ref{sec:algorithmic}.

\medbreak

\underline{\emph{On the proof of (\ref{main-1})$\Rightarrow$(\ref{main-4}), see Section~\ref{section:me_to_graph}.}} An important tool in the study of RAAGs is the concept of a \deffont{parabolic subgroup}: if $\Theta\leq\Gamma$ is an induced subgraph, then the natural homomorphism $A(\Theta)\to A(\Gamma)$ induced by the graph inclusion is injective. Its image, as well as all its conjugates in $A(\Gamma)$, are called parabolic subgroups of $A(\Gamma)$.

Up to replacing $\Delta$ by $\clr(\Delta)$, we will assume that $\Delta$ is clique reduced. Assume for a moment that $A(\Delta)$ embeds as a subgroup (and not just measurably) in $A(\Gamma)$, through an injective homomorphism $f\colon A(\Delta)\to A(\Gamma)$.

Given a vertex $u\in V(\Delta)$, identified with a generator of $A(\Delta)$, let $P_u$ be the smallest parabolic subgroup of $A(\Gamma)$ that contains $f(u)$. Using the fact that $\Gamma$ has no induced squares, one can then check that the assignment $u\mapsto P_u$ satisfies the following properties (see Remark~\ref{rem:algebraic_embeddings} and Lemma~\ref{lemma nonhat has clique support} for details):
\begin{enumerate}[label={(\alph*)}]
    \item\label{witness-intro-1} if $u$ and $v$ are adjacent in $\Delta$, then $P_u$ and $P_v$ normalise each other;
    \item\label{witness-intro-2} the converse also holds provided one of $P_u,P_v$ is abelian;
    \item\label{witness-intro-3} for every vertex $u\in V(\Delta)$ whose star is not a clique (i.e.\ whose normaliser in $A(\Delta)$ is non-abelian), $P_u$ is abelian and its normaliser is non-abelian.
\end{enumerate}
From there, if no vertex of $\Delta$ has a star which is a clique, then every $P_u$ is abelian, hence naturally corresponds to a clique of $\Gamma\ext$; this yields the desired embedding $\Delta\leq\Gamma\ext_k$. The combinatorics are a bit more subtle for vertices whose star is a clique, which have to be dealt with separately; this is the content of Proposition~\ref{prop:witness-to-graph}.

The same strategy extends to the case where $A(\Delta)$ is only assumed to measurably embed in $A(\Gamma)$. In this case, as reviewed in Section~\ref{section meas embeddings}, we have a measured groupoid equipped with two cocycles towards $A(\Delta)$ and $A(\Gamma)$. To run the above strategy, we rely on work of Huang and the third-named author \cite{Horbez_Huang}, which provides a notion of parabolic support for any measured groupoid equipped with a cocycle towards a RAAG. This enables, from a measurable embedding from $A(\Delta)$ to $A(\Gamma)$, to get an assignment $u\mapsto P_u$ that satisfies the same conditions~\ref{witness-intro-1},~\ref{witness-intro-2},~\ref{witness-intro-3} as above (see Lemma~\ref{lemma nonhat has clique support} and Proposition~\ref{prop:parabolic_support}), which is all we need to complete the proof.  

Let us mention here that under the stronger assumption that $A(\Gamma)$ and $A(\Delta)$ are two-dimensional and $\Gamma$ is not a star graph (as in Theorem~\ref{thm:intro-meas-algebraic}), whenever $u\in V(\Delta)$ is a vertex whose star is not a clique, the subgroup $P_u$ is cyclic; using this, we get an embedding $\Delta\leq\Gamma\ext$, without having to pass to the clique graph of $\Gamma\ext$, see Proposition~\ref{prop:witness-to-graph-warmup}. By the work of Kim--Koberda, this is enough to derive that $A(\Delta)\leq A(\Gamma)$ in this case.
\medbreak

\underline{\emph{On the proof of (\ref{main-4})$\Rightarrow$(\ref{main-3}), see Section~\ref{section:graph_to_algebraic}.}} Start with an embedding $\Delta\leq\Gamma\ext_k$ as an induced subgraph. Every vertex $K\in V(\Delta)$ corresponds to a clique in $\Gamma\ext$, which in turns yields an abelian parabolic subgroup $P_K$ of $A(\Gamma)$. For every such $K$, choose an element $g_K\in P_K$ with full support, i.e.\ not contained in any proper parabolic subgroup of $P_K$; we will write $\supp(g_k)=P_K$. The assignment $K\mapsto g_K$ then yields a homomorphism $A(\Delta)\to A(\Gamma)$, and the crucial question when trying to understand embeddings between right-angled Artin groups is the following: when is this homomorphism injective? In the case where every clique $K$ is reduced to one vertex, Kim and Koberda prove that, up to replacing each $g_K$ by a power $g_K^M$, this homomorphism is an embedding \cite{Kim_Koberda}. When $\Gamma$ has no induced squares, we extend their conclusion to the case where the cliques $K$ corresponding to vertices of $\Delta$ are pairwise disjoint. 

\begin{theo}[RAAG subgroups of RAAGs, see Theorem~\ref{thm:graph_to_algebraic_explicit}]{thm:main-2}
Let $\Gamma$ be a graph with no induced squares, and let $\Delta\leq \Gamma\ext_k$ be a finite graph whose vertices correspond to pairwise disjoint cliques of $\Gamma\ext$. Then there exists an integer $M>0$ with the following property: for any family $(g_K)_{K\in V(\Delta)}$ of elements of $A(\Gamma)$ such that $\supp(g_K)=P_K$ for all $K$, the map
\[\begin{aligned}
V(\Delta)&\to A(\Gamma)\\
K&\mapsto g_K^{M}
\end{aligned}\]
extends to an embedding $A(\Delta)\hookrightarrow A(\Gamma)$.
\end{theo}

In fact, in Lemma~\ref{lem:subgraph_disjoint_injective}, we prove that, under the stronger assumption $\Delta\leq \Gamma_k$, we can actually take $M=1$. The statement is a slightly simplified version; Theorem~\ref{thm:graph_to_algebraic_explicit} actually allows some induced squares, provided they do not interfere too much with the cliques given by $\Delta$. However, neither the assumption that $\Gamma$ has no induced squares, nor that the vertices of $\Delta$ correspond to pairwise disjoint cliques of $\Gamma\ext$, can be dropped entirely, as shown by the examples in Remark~\ref{rem:necessary_assumptions_algebraic}.

Our proof of Theorem~\ref{thm:main-2} relies on reduced forms for elements of RAAGs, in an argument inspired by disc digrams in $\mathrm{CAT}(0)$ cube complexes (Lemma~\ref{lem:intersection_2_convex}). It also exploits in Lemma~\ref{lem:inductive_step} the decomposition of $A(\Gamma)$ as an HNN extension \[A(\Gamma)\simeq A(\Gamma\setminus\{v\})\ast_{A(\lk_\Gamma(v))}\] in order to run an inductive argument, adding all vertices from $\Delta$ one at a time. Some care is required, as the HNN splittings only permit to adjoin to $\Delta$ vertices corresponding to cliques with a single element. We get around this difficulty by introducing an intermediate RAAG $A(\widetilde \Delta)$ constructed in such a way that we can use the one-element clique case to obtain injections from $A(\Delta)$ to $A(\widetilde \Delta)$ and from $A(\widetilde \Delta)$ to $A(\Gamma)$ (see the proof of Lemma~\ref{lem:subgraph_disjoint_injective}).

Back to our proof that (\ref{main-4})$\Rightarrow$(\ref{main-3}), if $\clr(\Delta)\leq\Gamma\ext_k$, then by blowing-up $\Gamma$ to a graph $\widetilde{\Gamma}$ such that $A(\widetilde{\Gamma})$ is a graph product of free abelian groups $\mathbb{Z}^n$ over $\Gamma$, one can ensure that $\Delta\leq\widetilde\Gamma\ext_k$ with the extra property that vertices of $\Delta$ represent pairwise disjoint cliques of $\widetilde\Gamma\ext$ (in fact, Lemma~\ref{lem:disjoint_cliques} is strong enough to do so with $n=|V(\Delta)|$, proving that Theorem~\ref{thm:small_clique_expansion} follows from Theorem~\ref{thm:main} in the special case where $\Gamma$ has no squares). We insist that even if $\Delta\leq\Gamma\ext_k$ to start with (without passing to the clique reduction), this blow-up is in general necessary to get the disjointness property, see Remark~\ref{rem:necessary_expansion}. It then follows from Theorem~\ref{thm:main-2} that $A(\Delta)$ embeds in $A(\widetilde\Gamma)$, as desired.

\paragraph*{Organisation of the paper.} After a preliminary section on graphs and right-angled Artin groups (Section~\ref{sec:background-raags}), we establish some useful facts about extension graphs and their clique graphs in Section~\ref{section:extension_graphs}, already proving Theorem~\ref{thm:small_clique_expansion}. In Section~\ref{section:background_me}, we review the notions of mesurable embeddings and measured groupoids. We then prove the implication (\ref{main-1})$\Rightarrow$(\ref{main-4}) of Theorem~\ref{thm:main} in Section~\ref{section:me_to_graph}, and prove Theorem~\ref{thm:main-2} in Section~\ref{section:graph_to_algebraic}. We finally summarise the proofs of Theorems~\ref{thm:main} and~\ref{thm:intro-meas-algebraic} and give some examples and corollaries in Section~\ref{sec:conclusion}.

\paragraph*{Acknowledgments.}

We thank Jean Lécureux for many fruitful discussions. We thank Naomi Andrew for her useful advice, Sam Fisher for suggesting Corollary~\ref{cor:coherent} to us, and Harry Petyt for pointing us to \cite{Rul}. 

\section{Background on graphs and right-angled Artin groups}\label{sec:background-raags}
\subsection{Graphs}

All graphs will be simplicial (unoriented, no loops or multi-edges), and endowed with their usual path metric. \deffont{Graph maps} send vertices to vertices and edges to edges or vertices. \deffont{Graph morphisms} send vertices to vertices and edges to edges. For a graph $\Gamma$, we denote by $V(\Gamma)$ its vertex set, and by $E(\Gamma)$ its edge set. A subgraph $\Delta$ of $\Gamma$ is \deffont{induced} if any two vertices of $\Delta$ are adjacent in $\Delta$ if and only if they are adjacent in $\Gamma$. We write $\Delta\leq\Gamma$ to mean that $\Delta$ embeds as an induced subgraph of $\Gamma$.

\begin{defn}
\label{de:hat}Let $\Gamma$ be a graph and $v$ be a vertex of $\Gamma$. The \deffont{link} and \deffont{star} of $v$, denoted respectively by $\lk_\Gamma(v)$ and $\st_\Gamma(v)$, are the subgraphs of $\Gamma$ induced respectively by the set of vertices at distance $1$ from $v$, and the set of vertices at distance at most $1$ from $v$. The \deffont{link-equivalence class} of $v$ is the set $[v]_{\lk}$ of vertices $w$ such that $\lk_\Gamma(w)=\lk_\Gamma(v)$.

For $\Delta\leq \Gamma$, we denote by $\Delta^\perp$ the (possibly empty) subgraph induced by the intersection of $\lk_\Gamma(v)$ over all vertices $v$ of $\Delta$. In other words, $\Delta^\perp$ is the largest induced subgraph of $\Gamma$ spanning a join with $\Delta$.
\end{defn}

%\medbreak
\begin{defn}
A graph map $f\colon \Gamma \to \Delta$ is a \deffont{clique map} if $f$ is surjective and for any vertices $v,w$ of $\Gamma$ such that $f(v)=f(w)$, we have $\st_\Gamma(v)=\st_\Gamma(w)$. The \deffont{order} of $f$ is the (possibly infinite) cardinality of the largest vertex preimage. We say that $\Gamma$ is a \deffont{clique expansion} (of order $o$) of $\Delta$ if there exists a clique map $f\colon \Gamma \to \Delta$ (of order $o$).

Following \cite{Horbez_Huang_2} and \cite[Section~9]{Escalier_Horbez}, a graph $\Delta$ is \deffont{clique reduced} if distinct vertices of $\Delta$ have distinct stars, i.e.~the only clique maps with domain $\Delta$ are isomorphisms. Every graph $\Gamma$ has a clique map to a clique-reduced graph, and this graph is unique up to isomorphism; we call it the \deffont{clique reduction} of $\Gamma$ and denote it $\clr(\Gamma)$.
\end{defn}

Note that we always have $\clr(\Gamma)\leq \Gamma$ since any choice of one vertex in each vertex preimage provides a section of the clique map. Note also that all the vertex preimages, edge preimages, and more generally clique preimages of a clique map are cliques.
\medskip

\begin{lem}
\label{lem:square_expansion}Let $\Gamma$, $\Delta$ be graphs, with $\Delta$ clique-reduced. Let $\Gamma'$ be a clique expansion of $\Gamma$. If $\Delta\leq \Gamma'$, then $\Delta\leq \Gamma$.
\end{lem}
\begin{proof}
Let $f\colon \Gamma'\to \Gamma$ be a clique map. Assume $\Delta\leq \Gamma'$. Since edge preimages under $f$ are cliques, $f(\Delta)\leq \Gamma$. The restriction $\Delta\to f(\Delta)$ of $f$ is clearly a clique map. Thus, $\Delta\simeq f(\Delta)\leq \Gamma$, since $\Delta$ is clique reduced.
\end{proof}

The following definition appears in work of Kim--Koberda \cite{Kim_Koberda}.
\smallskip

\begin{defn}
Let $\Gamma$ be a graph. The \deffont{clique graph} of $\Gamma$ is the graph $\Gamma_k$ having one vertex for each non-empty finite clique of $\Gamma$, and one edge for each pair of cliques which are both contained in the same clique of $\Gamma$.
\end{defn}

\medskip
\begin{rem}\label{rem:clique-graph}
Note that $\Gamma\leq \Gamma_k$, by sending every vertex $v$ to the one-element clique $\{v\}$. Note also that if $\Delta\leq\Gamma$, then $\Delta_k\leq\Gamma_k$.
\end{rem}
\medskip

\begin{rem}\label{rem:clique-map}
Any clique map $f\colon \Gamma\to \Delta$ induces a clique map $f_k\colon \Gamma_k\to \Delta_k$, as follows. Given a finite non-empty clique $K = \{k_1,\dots, k_n\}$ of $\Gamma$, we set $f_k(K)\coloneqq \{f(k_1),\dots, f(k_n)\}$, which is a finite and non-empty clique of $\Delta$. If $K_1, K_2$ are adjacent vertices of $\Gamma_k$, they are contained in a common clique $K$, hence $f_k(K_1), f_k(K_2)$ are contained in the common clique $f_k(K)$, thus they are adjacent or equal in $\Delta_k$. Therefore, $f_k$ is a graph map. Given $L = \{\ell_1,\dots, \ell_m\}$ a finite non-empty clique of $\Gamma$, since $f$ is surjective on vertices, there exist $k_1,\dots, k_m$ vertices of $\Gamma$ mapped by $f$ to $\ell_1,\dots, \ell_m$ respectively. Since cliques preimages of the clique map $f$ are cliques, $K\coloneqq \{k_1,\dots, k_m\}$ is a clique and $f_k(K)=L$: the map $f_k$ is surjective on vertices. The same argument proves that $f_k$ is surjective on edges. To conclude, let $K_1, K_2$ be vertices of $\Gamma_k$, and assume that $f_k(K_1) = f_k(K_2)$. Let $K_3$ be a vertex of $\Gamma_k$ adjacent or equal to $K_1$. Then $f_k(K_2)$ and $f_k(K_3)$ are adjacent or equal, hence contained in a common clique $L$ of $\Gamma$. This means that $K_2$ and $K_3$ are both contained in $f^{-1}(L)$ which is a clique of $\Gamma$ since $f$ is a clique map: $K_2$ and $K_3$ are adjacent or equal. Since $K_3$ was arbitrary, $\st_{\Gamma_k}(K_1) = \st_{\Gamma_k}(K_2)$ and $f_k$ is a clique map.
\end{rem}
\bigskip

\subsection{Right-angled Artin groups and graph products}
\medskip

Recall that, given a graph $\Gamma$ and a family of groups indexed by $V(\Gamma)$, $(G_v)_{v\in V(\Gamma)}$, one defines the \deffont{graph product} of $(G_v)$ over $\Gamma$ as the following group:
\[\bigast_{v\in V\Gamma}G_v \,/ \ngen{ [g_v,g_w] \mid \{v,w\}\in E(\Gamma),\,  g_v\in G_v,\, g_w\in G_w}.\] 

The \deffont{right-angled Artin group}, or RAAG, associated to $\Gamma$, denoted by $A(\Gamma)$, is the graph product over $\Gamma$ of a family consisting only of infinite cyclic groups. As usual, we consider the infinite cyclic group $G_v$ as generated by the element $v\in V(\Gamma)$, so that $V(\Gamma)$ appears as a generating subset of $A(\Gamma)$, called the set of \deffont{standard generators}.

\bigskip 
\begin{defn}
\label{defn:reduction}A \deffont{reduced form} for an element $g\in A(\Gamma)$ is a word $w$ in the standard generators of $A(\Gamma)$ and their inverses representing $g$ and having minimal length among words with this property (such a reduced form is not unique in general).

Let $w$ and $w'$ be two words whose letters are elements of $V(\Gamma)$ and their inverses. Assume that $w$ and $w'$ represent the same element $g$ of $A(\Gamma)$, and that $w$ is a reduced form. A \deffont{reduction} of $w'$ into $w$ is a sequence of words starting at $w'$ and ending at $w$, where each word differs from the next by an \deffont{elementary move}: either the exchange of two adjacent letters commuting in $A(\Gamma)$ and distinct even up to inversion, or the deletion of a two-letter subword of the form $vv^{-1}$ or $v^{-1}v$.
\end{defn}

Note that every word of the reduction still represents $g$. Such a reduction always exists for any choice of $w, w'$ by work of Hermiller and Meier \cite{Hermiller_Meier}. This can be used to prove that RAAGs are torsion free. Given a word $w = v_1^{\epsilon_1}\dots v_n^{\epsilon_n}$ with $v_i\in V(\Gamma)$ and $\epsilon_i=\pm 1$, we define $w^{-1}\coloneqq v_n^{-\epsilon_n}\dots v_1^{-\epsilon_1}$ as usual.

By a result of Droms \cite{Droms_isomorphism}, two RAAGs $A(\Gamma)$ and $A(\Delta)$ are isomorphic if and only if $\Gamma$ and $\Delta$ are isomorphic as graphs, making the defining graph of a RAAG a well-defined object up to isomorphism. A RAAG splits as a non-trivial direct (resp. free) product if and only if its defining graph splits as a non-trivial join (resp. is disconnected).

\begin{rem}
\label{rem:graph_product_expansion}
The class of right-angled Artin groups is closed under graph products. A group splits as a graph product of non-trivial free abelian groups (of rank at most $n$) over a graph $\Gamma$ if and only if it is a RAAG whose defining graph is a clique expansion of $\Gamma$ (of order at most $n$).
\end{rem}

Recall also that for every induced subgraph $\Delta\leq \Gamma$, the natural group homomorphism $A(\Delta)\to A(\Gamma)$ is injective, allowing us to identify $A(\Delta)$ with a subgroup of $A(\Gamma)$. In other words, $\Delta\leq \Gamma$ implies $A(\Delta)\leq A(\Gamma)$, but the converse is false (for example when $A(\Gamma)$ is a free group). Still with $\Delta \leq \Gamma$, a word $w$ in reduced form (using letters in $V(\Gamma)$ and their inverses) representing an element of $A(\Delta)$ contains only letters in $V(\Delta)$ and their inverses, because reducing a word on $V(\Delta)$ into $w$ does not introduce new letters. Consequently, for $\Delta_1,\Delta_2\leq \Gamma$, we have $A(\Delta_1)\cap A(\Delta_2) = A(\Delta_1\cap\Delta_2)$. A subgroup of $A(\Gamma)$ is \deffont{parabolic} if it is conjugate to $A(\Delta)$ for some $\Delta\leq \Gamma$. The class of parabolic subgroups is closed under intersections \cite[Proposition~2.6]{DKR}, \cite[Corollary~3.6]{Antolin_Minasyan}.

For $\Gamma$ a graph and $g\in A(\Gamma)$, the \deffont{(parabolic) support} of $g$ is the smallest parabolic subgroup of $A(\Gamma)$ containing $g$, denoted $\supp(g)$.

Given $\Delta\leq\Gamma$, the normaliser of $A(\Delta)$ in $A(\Gamma)$ is equal to the (internal) direct product $A(\Delta)\times A(\Delta^\perp)$ by \cite[Proposition~2.2(1)]{CCV}, \cite[Proposition~3.13]{Antolin_Minasyan}. Let now $P$ be a parabolic subgroup of $A(\Gamma)$: there exist $g\in A(\Gamma)$ and $\Delta\leq \Gamma$ such that $P = gA(\Delta)g^{-1}$. The subgraph $\Delta$ is uniquely determined and called the \deffont{type} of $P$ by \cite[Proposition~2.2(2)]{CCV}, \cite[Corollary~3.8]{Antolin_Minasyan}. The element $g$ is uniquely determined up to right-multiplication by an element of $N_{A(\Gamma)}(A(\Delta))=A(\Delta)\times A(\Delta^\perp)$. As a consequence, in the case where $\Delta$ is a single vertex, no two distinct standard generators (elements of $V(\Gamma)$) are conjugate in $A(\Gamma)$.

\begin{rem}\label{rem:parabolic}
Given $\Delta\leq\Gamma$, the parabolic subgroup $A(\Delta)$ is also a right-angled Artin group (over $\Delta$), so one can consider parabolic subgroups of $A(\Delta)$ for this structure: these are exactly the subgroups of the form $hA(\Theta)h^{-1}$ with $\Theta\leq\Delta$ and $h\in A(\Delta)$. In particular, every parabolic subgroup of $A(\Delta)$ is a parabolic subgroup of $A(\Gamma)$.

Conversely, every parabolic subgroup of $A(\Gamma)$ which is contained in $A(\Delta)$, is also a parabolic subgroup of $A(\Delta)$. Indeed, by \cite[Lemma~3.7]{Antolin_Minasyan}, for every $\Theta\leq\Gamma$ and $g\in A(\Gamma)$, if $gA(\Theta)g^{-1}\leq A(\Delta)$, then there exists $h\in A(\Delta)$ such that $gA(\Theta)g^{-1}=h A(\Theta)h^{-1}$. 

More generally, if $P=gA(\Delta)g^{-1}$ for some $g\in A(\Gamma)$, then parabolic subgroups of $A(\Gamma)$ that are contained in $P$ are exactly the subgroups of the form $gQg^{-1}$, where $Q\subseteq A(\Delta)$ is a parabolic subgroup of $A(\Delta)$, for its structure of right-angled Artin group over $\Delta$.
\end{rem}

\begin{defn}
\label{def:support_perp}Let $P$ be a parabolic subgroup of $A(\Gamma)$ of type $\Delta$, and write $P = gA(\Delta)g^{-1}$. Define $P^\perp$ to be the parabolic subgroup $gA(\Delta^\perp)g^{-1}$, which does not depend on the choice of $g$ in its coset of $N_{A(\Gamma)}(A(\Delta))$, since the latter normalises $A(\Delta^\perp)$. Note that $P$ and $P^\perp$ span a direct product in $A(\Gamma)$ and $N_{A(\Gamma)}(P) = P\times P^\perp$.

With the same notations, if $\Delta$ has a maximal join decomposition as the join of $\Delta_1, \dots , \Delta_n$ (such a decomposition always exists if $\Delta$ is finite), define the \deffont{maximal product decomposition} of $P$ to be the (internal) direct product of parabolics $P = gA(\Delta_1)g^{-1} \times \dots \times gA(\Delta_n)g^{-1}$. This does not depend on the choice of $g$ because $N_{A(\Gamma)}(A(\Delta)) = A(\Delta)\times A(\Delta^\perp) = A(\Delta_1)\times \dots \times A(\Delta_n)\times A(\Delta^\perp)$ normalises each of the $A(\Delta_i)$.
\end{defn}

\begin{thm}[{Servatius \cite[Centraliser~theorem]{Servatius}}]
\label{thm:centraliser}Let $\Gamma$ be a graph and $g\in A(\Gamma)$. Let $P_1\times \dots \times P_n$ be the maximal product decomposition of $\supp(g)$, and write $g$ as a product $g_1^{\alpha_1}\cdots g_n^{\alpha_n}$, where $g_i\in P_i$ and $g_i$ is not a proper power. Then the centraliser of $g$ in $A(\Gamma)$, which is also the normaliser of $\gen{g}$, splits as the following (internal) direct product:
\[Z_{A(\Gamma)}(g)= \gen{g_1}\times\dots\times \gen{g_n}\times \supp(g)^\perp.\]
\end{thm}
\medskip

In particular, we always have the inclusion
\[Z_{A(\Gamma)}(g) \subseteq \supp(g)\times \supp(g)^\perp,\]
and this inclusion becomes an equality if $\supp(g)$ is abelian, because the $P_i$ are cyclic in that case. In the special case where $g$ is equal to a standard generator $v\in V(\Gamma)$, this becomes $Z_{A(\Gamma)}(v) = A(\st_\Gamma(v))$.

\subsection{Extension graphs} The following definition also appears in the article of Kim and Koberda \cite{Kim_Koberda}.
\begin{defn}
Let $\Gamma$ be a graph. The \deffont{extension graph} of $\Gamma$ is the graph $\Gamma\ext$ with vertex set $\{gvg^{-1}\mid g \in A(\Gamma), v \in V(\Gamma)\}$, where two vertices are joined by an edge if and only if they are distinct, commuting elements of $A(\Gamma)$. The group $A(\Gamma)$ has a natural action by conjugation on $\Gamma\ext$. For consistency of notation, given a subgraph $\Delta$ of $\Gamma\ext$ and $g\in A(\Gamma)$, we will denote by $g\Delta g^{-1}$ its $g$-translate under this action, and say the subgraphs $\Delta$ and $g\Delta g^{-1}$ are \deffont{conjugate}.

We use the notation $\Gamma\ext_k$ to denote the clique graph of $\Gamma\ext$ (which is different in general from $(\Gamma_k)\ext$).
\end{defn}

Even when $\Gamma$ is finite, $\Gamma\ext$ is in general a locally infinite graph. It has a purely graph-theoretic but quite involved definition as the limit of a nested sequence of finite graphs starting with $\Gamma$ and iteratively doubling the graph over the star of some vertex, visiting each of the vertices infinitely many times (see \cite[Lemma~3.1]{Kim_Koberda}). Note that we always have $\Gamma\leq \Gamma\ext$.

\medskip
We will further discuss extension graphs in Section~\ref{section:extension_graphs}, but we record already some elementary properties from \cite{Kim_Koberda} (which are proved only in the case where $\Gamma$ is finite, but the proofs work directly in the infinite case).
\medskip

\begin{lem}[{\cite[Lemma~3.5]{Kim_Koberda}}]
\label{lem:extension_properties}The following hold for every graph $\Gamma$:
\begin{itemize}
    \item If $\Gamma$ splits as a join $\Delta * \Theta$, then $\Gamma\ext$ splits as a join $\Delta\ext * \Theta\ext$.
    \item If $\Gamma$ splits as a disjoint union $\Delta \sqcup \Theta$ of two non-empty subgraphs, then $\Gamma\ext$ splits as a disjoint union of infinitely many connected components, which are all isomorphic to $\Delta\ext$ or $\Theta\ext$.
    \item For any graph $\Delta$, if $\Delta\leq \Gamma$, then $\Delta\ext\leq \Gamma\ext$.
    \item If $\Gamma$ is connected, then $\Gamma\ext$ is connected.
    \item If $\Gamma$ is finite and connected, then $\Gamma\ext$ has finite diameter if and only if $\Gamma$ splits as a non-trivial join or consists of a single vertex.
\end{itemize}
\end{lem}
\medskip

We also record the main result of Kim and Koberda's article \cite{Kim_Koberda}.
\medskip

\begin{thm}[{\cite[Theorem~1.3 and Theorem~1.4]{Kim_Koberda}}]
\label{thm:main_Kim_Koberda}Let $\Delta$, $\Gamma$ be finite graphs. If $\Delta\leq \Gamma\ext$, then $A(\Delta)\leq A(\Gamma)$. Conversely, if $A(\Delta)\leq A(\Gamma)$, then $\Delta\leq \Gamma\ext_k$.
\end{thm}

\begin{rem}
\label{rem:infinite_Kim_Koberda}
Note that the assumption in Theorem~\ref{thm:main_Kim_Koberda} that $\Gamma$ is finite is unnecessary. Indeed, assume $\Delta$ is finite but $\Gamma$ is not. If $\Delta\leq \Gamma\ext$, then each vertex of $\Delta$ corresponds to an element of $A(\Gamma)$ (conjugate to a standard generator). These finitely many elements are all contained in a finitely generated parabolic subgroup $A(\Gamma')$ with $\Gamma'\leq \Gamma$ finite. This means that $\Delta\leq \Gamma'\ext$. By Theorem~\ref{thm:main_Kim_Koberda}, $A(\Delta)\leq A(\Gamma')\leq A(\Gamma)$.

Conversely, if $A(\Delta)\leq A(\Gamma)$, then finitely many elements of $A(\Gamma)$ generate a subgroup isomorphic to $A(\Delta)$. These finitely many elements are all contained in a finitely generated parabolic subgroup $A(\Gamma')$ with $\Gamma'\leq \Gamma$ finite. This means that $A(\Delta)\leq A(\Gamma')$. By Theorem~\ref{thm:main_Kim_Koberda}, we have $\Delta\leq \Gamma'\ext_k$. By Lemma~\ref{lem:extension_properties}, we have $\Gamma'\ext\leq\Gamma\ext$, and hence $\Gamma'\ext_k\leq \Gamma\ext_k$ (Remark~\ref{rem:clique-graph}). Altogether $\Delta\leq\Gamma\ext_k$.
\end{rem}
\medskip

Observation of the proof of \cite[Theorem~1.3]{Kim_Koberda} allows easily to extract a more precise statement, which we will use in Section~\ref{section:graph_to_algebraic}.
\smallskip

\begin{thm}[from {\cite[page 513, first proof of Theorem~1.3]{Kim_Koberda}}]
\label{thm:powers_Kim_Koberda}Let $\Gamma$ be a graph, and $\Delta$ a finite induced subgraph of $\Gamma\ext$. Then, there exists an integer $M>0$ such that the map
\[\begin{aligned}
V(\Delta) &\to A(\Gamma)\\
g &\mapsto g^M
\end{aligned}\]
extends to an embedding $A(\Delta)\hookrightarrow A(\Gamma)$.
\end{thm}
\medskip

There is also a more precise statement used by Kim and Koberda to derive the converse direction.
\smallskip

\begin{thm}[{\cite[Theorem~4.3]{Kim_Koberda}}]
\label{thm:clique_Kim_Koberda}Let $\Gamma$ be a graph and $\Delta$ be a finite graph such that $A(\Delta)\leq A(\Gamma)$. Then there exists an embedding $\varphi\colon A(\Delta)\to A(\Gamma\ext)$ such that for each vertex $v$ of $\Delta$, the element $\varphi(v)$ is a product of pairwise commuting standard generators of $A(\Gamma\ext)$.
\end{thm}
The assumption of finiteness of $\Gamma$ made in \cite{Kim_Koberda} is once again unneccessary since $A(\Delta)$ embeds in some parabolic subgroup $A(\Gamma')$ with $\Gamma'\leq \Gamma$ finite as before.

\bigskip
Finally, Escalier and the third-named author defined the following generalisation of the extension graph to the case of graph products.

\begin{defn}[{\cite[Definition~2.10]{Escalier_Horbez}}]
Let $\Gamma$ be a graph and let $(G_v)_{v\in V(\Gamma)}$ be a family of (non-trivial) groups. Let $G$ denote the graph product of $(G_v)$ over $\Gamma$. The \deffont{extension graph of the graph product} of $(G_v)$ over $\Gamma$ is the graph with vertex set $\{gG_vg^{-1}\mid g\in G, v\in V(\Gamma)\}$, where two vertices are joined by an edge if and only if they are distinct and commute (meaning that all elements of one commute with all elements of the other). As in the case of RAAGs, the group $G$ acts by conjugation on this graph.
\end{defn}

\begin{lem}[{\cite[Lemma~2.12]{Escalier_Horbez}}]
\label{lem:extension_graph_product}Let $\Gamma$ be a graph and let $(G_v)_{v\in V(\Gamma)}$, $(H_v)_{v\in V(\Gamma)}$ be two families of non-trivial groups such that for each $v\in V(\Gamma)$, the groups $G_v$ and $H_v$ have the same cardinality. Then the extension graphs of the graph products of $(G_v)$ and $(H_v)$ over $\Gamma$ are isomorphic.
\end{lem}
The result is only proved for $\Gamma$ finite, but the proof remains the same for $\Gamma$ infinite. Note in particular that if all the $G_v$ are countably infinite, by applying the lemma with $H_v \simeq \ZZ$ for all $v$, we get that the extension graph of the graph product of $(G_v)$ over $\Gamma$ is isomorphic to $\Gamma\ext$.

\section{Properties of extension graphs}
\label{section:extension_graphs}
In this section, we establish further properties of extension graphs, proving auxiliary lemmas for the next sections.

\subsection{Induced squares}

From Theorem~\ref{thm:main_Kim_Koberda}, Kim and Koberda recover the following result due to Kambites \cite{Kambites}, which we state for a possibly infinite graph for generality. We also add a remark about the clique graph.

\begin{cor}
\label{cor:square_F2xF2}Let $\Gamma$ be a graph. The following are equivalent:
\begin{enumerate}
    \item\label{square1} $A(\Gamma)$ contains $F_2\times F_2$.
    \item\label{square2} $\Gamma\ext$ contains an induced square.
    \item\label{square3} $\Gamma$ contains an induced square.
    \item\label{square4} $\Gamma_k$ contains an induced square.
\end{enumerate}
More precisely, for every induced square of $\Gamma_k$ (whose vertices correspond to cliques of $\Gamma$), there is a choice of one vertex in each clique such that these four vertices span an induced square of $\Gamma$.
\end{cor}
\begin{proof}
If (\ref{square1}) holds, like in Remark~\ref{rem:infinite_Kim_Koberda}, some $A(\Gamma')$ contains $F_2\times F_2$, for a finite $\Gamma'\leq \Gamma$. By \cite[Corollary~4.4]{Kim_Koberda}, $\Gamma'\ext$ contains an induced square. Thus, so does $\Gamma\ext\geq \Gamma'\ext$ and (\ref{square2}) holds.

If (\ref{square2}) holds, like in Remark~\ref{rem:infinite_Kim_Koberda}, some $\Gamma'\ext$ contains an induced square, for a finite $\Gamma'\leq \Gamma$. By \cite[Lemma~3.9~(2)]{Kim_Koberda}, $\Gamma'$ contains an induced square. Thus, so does $\Gamma$ and (\ref{square3}) holds.

Now, (\ref{square3})$\Rightarrow$(\ref{square1}) comes directly from the fact that $\Delta\leq \Gamma$ implies $A(\Delta)\leq A(\Gamma)$.

For (\ref{square3})$\Rightarrow$(\ref{square4}), note that if four vertices $v_1,v_2,v_3,v_4$ induce a square in $\Gamma$, then the one-element cliques $\{v_1\},\{v_2\},\{v_3\},\{v_4\}$ induce a square in $\Gamma_k$.

Finally, for (\ref{square4})$\Rightarrow$(\ref{square3}), let $K_1$, $K_2$, $K_3$, $K_4$ be cliques of $\Gamma$ spanning a square in $\Gamma_k$ in that order. Since $K_1$ and $K_3$ (resp. $K_2$ and $K_4$) are not adjacent in $\Gamma_k$, there exist $x_1\in V(K_1)$ and $x_3\in V(K_3)$ (resp. $x_2\in V(K_2)$ and $x_4\in V(K_4)$) which are distinct and not adjacent. The vertices $x_1$ and $x_2$ of $\Gamma$ cannot be equal, because $x_1$ is equal or adjacent to $x_4$ and $x_2$ is not. Therefore $x_1$ and $x_2$ are adjacent. By the same argument for the three remaining pairs, $x_1$, $x_2$, $x_3$, $x_4$ induce a square of $\Gamma$. This shows that (\ref{square4})$\Rightarrow$(\ref{square3}).
\end{proof}
\bigskip

\subsection{Extended types}

\begin{lem}
\label{lem:extension_cliques}Let $\Gamma$ be a graph. Every finite clique in $\Gamma\ext$ is of the form $gKg^{-1}$ for some $g\in A(\Gamma)$ and some clique $K$ of $\Gamma$.
\end{lem}
\begin{proof}
The proof is by induction on the size of the clique. It is clear for the empty clique. Consider a clique with vertices $v_0,\dots, v_n \in V(\Gamma\ext)$, and assume by the inductive hypothesis, up to conjugating by some element of $G$, that $v_1,\dots, v_n$ are in $V(\Gamma)$. Write $v_0=gvg^{-1}$ with $v\in V(\Gamma)$ and $g\in A(\Gamma)$. By Theorem~\ref{thm:centraliser}, $gvg^{-1}$ is in \[\displaystyle\bigcap_{i=1}^n A(\st_\Gamma(v_i)) = A\left(\bigcap_{i=1}^n \st_\Gamma(v_i)\right).\] By \cite[Proposition~2.1(2)]{CCV}, applied with $\Theta=\{v\}$ and $\Lambda=\displaystyle\bigcap_{i=1}^n\st_\Gamma(v_i)$, the vertex $v$ is adjacent to $v_1,\dots,v_n$, and we have $g=g_1g_2$ for some $g_1$ commuting with $v_1,\dots,v_n$, and some $g_2$ commuting with $v$. It follows that $gvg^{-1}=g_1vg_1^{-1}$, and conjugating the whole clique by $g_1^{-1}$ fixes $v_1,\dots,v_n$ and maps $v_0$ to $v\in V(\Gamma)$, completing the inductive step.
\end{proof}
\medskip

\begin{rem}
\label{rem:extension_colour}As a consequence of Lemma~\ref{lem:extension_cliques}, $\Gamma\ext$ is covered by the copies of $\Gamma$ given by $g\Gamma g^{-1}$ for $g\in A(\Gamma)$. Moreover, every $n$-colouration of $\Gamma$ induces an $n$-colouration of $\Gamma\ext$ since distinct conjugates of the same standard generator are never adjacent (this statement already appeared in \cite[Lemma~3.5]{Kim_Koberda}).
\end{rem}
\medskip

We will need the following definition, adapting the notion of type of a parabolic subgroup to the extension graph. 
\begin{defn}
\label{def:xtype}Let $\Gamma$ be a graph, and let $P$ be a parabolic subgroup of $A(\Gamma)$. The \deffont{extended type} of $P$, denoted by $\xtype(P)$, is the subgraph of $\Gamma\ext$ spanned by all vertices associated with elements of $A(\Gamma)$ that are contained in $P$.
\end{defn}

\begin{rem}
Write $P=gA(\Delta)g^{-1}$ for some $\Delta\leq\Gamma$ and some $g\in A(\Gamma)$. In view of Remark~\ref{rem:parabolic}, cyclic parabolic subgroups of $A(\Gamma)$ contained in $P$ are exactly the subgroups of the form $gZg^{-1}$, where $Z$ is a cyclic parabolic subgroup of $A(\Delta)$, i.e.\ $Z=\langle hvh^{-1}\rangle$ for some $h\in A(\Delta)$ and some $v\in V(\Delta)$. In particular $\Delta\ext\leq\Gamma\ext$ by sending $hvh^{-1}$ to itself, and $\xtype(P)=g\Delta\ext g^{-1}$.  
\end{rem}
\medskip

\begin{rem}
\label{rem:xtypeclique}
Keping the notations from the previous remark, in the special case where $P$ is abelian of finite rank, $\Delta$ is a clique and $\Delta\ext = \Delta$. Note that, by Lemma~\ref{lem:extension_cliques}, the extended type establishes a one-to-one correspondence between finite cliques of $\Gamma\ext$ and finite-rank abelian parabolic subgroups of $A(\Gamma)$.
\end{rem}

\bigskip
The following lemma will be used in Section~\ref{section:me_to_graph} for our proof of the implication (\ref{main-1})$\Rightarrow$(\ref{main-4}) of Theorem~\ref{thm:main}.

\begin{lem}
\label{lemma:xtype}Let $\Gamma$ be a graph without induced squares and let $Q$ be a finitely generated abelian parabolic subgroup of $A(\Gamma)$ such that $Q^\perp$ is non-abelian. Let $(x_u)_{u \in U}$ be a finite family of vertices of $\Gamma\ext \setminus \xtype(Q\times Z(Q^\perp))$. Let $n$ be an integer. 

Then there exist vertices $v_1, \dots, v_n$ of $\xtype(Q^\perp)\setminus\xtype(Z(Q^\perp))$ such that
\begin{itemize}
    \item For $1\leq i<j\leq n$, the vertices $v_i$ and $v_j$ are distinct and non-adjacent. 
    \item For $1\leq i\leq n$ and $u\in U$, the vertices $v_i$ and $x_u$ are distinct and non-adjacent.
\end{itemize}
\end{lem}

\begin{proof}
Since $\Gamma$ has no induced square, if the type of $Q^\perp$ splits non-trivially as a join, then at most one of the join factors can contain two non-adjacent vertices. One can therefore decompose the type of $Q^\perp$ as a join $C\ast \Gamma_1$, where $C$ is a (possibly empty) clique, while $\Gamma_1$ is not a clique and does not decompose non-trivially as a join. This corresponds to a direct product decomposition $Q^\perp=Z(Q^\perp)\times Q_1$. As $Q^\perp$ is non-abelian, $\Gamma_1\neq\emptyset$. If a cyclic, parabolic subgroup of $A(\Gamma)$ is contained in $Q^\perp$, then it is contained either in $Z(Q^\perp)$ or in $Q_1$. This also yields a join decomposition \[\xtype(Q^\perp)=\xtype(Z(Q^\perp))\ast\xtype(Q_1).\]

Our goal is to construct pairwise distinct, pairwise non-adjacent vertices $v_1,\dots,v_n$ of $\xtype(Q_1)$, none of which is equal or adjacent to some $x_u$ with $u\in U$.

For $u\in U$, let $Z_{u}$ be the parabolic subgroup defined as the intersection of $Q_1$ with the centraliser of the cyclic parabolic subgroup associated to $x_u$. We now observe that $\xtype(Z_u)$ has bounded diameter when viewed as a subgraph of $\xtype(Q_1)$. Indeed: 
\begin{itemize}
    \item Assume first that $x_u\in V(\xtype(Q_1))$. Then $Z_u$ is the centraliser of $x_u$ in $Q_1$; in particular $\xtype(Z_u)$ is the star of $x_u$ in $\xtype(Q_1)$, so it has bounded diameter. 
    \item Assume now that $x_u\notin V(\xtype(Q_1))$. We will prove that $\xtype(Z_u)$ does not contain two non-adjacent vertices $x_1,x_2$ of $\xtype(Q_1)$, so assume for the sake of contradiction that it does. Since by assumption $x_u\notin\xtype(Z(Q^\perp))$, we have $x_u\notin\xtype(Q^\perp)$. This implies in particular that $Q^\perp\neq A(\Gamma)$ and hence $Q\neq\{1\}$. We can thus find a vertex $x_0\in\xtype(Q)$. Then $x_0x_1x_ux_2$ is an induced square in $\Gamma\ext$, and by Corollary~\ref{cor:square_F2xF2} there is also an induced square in $\Gamma$, a contradiction.  
\end{itemize}
The graph $\xtype(Q_1)$ has infinite diameter by Lemma~\ref{lem:extension_properties}, while there are only finitely many $u\in U$ giving bounded subsets $\xtype(Z_u)$. We can therefore find pairwise distinct, pairwise non-adjacent vertices $v_1,\dots,v_n\in V(\xtype(Q_1))\setminus\bigcup_{u\in U} V(\xtype(Z_u))$: they are all distinct from and non-adjacent to any of the vertices $x_u$ with $u\in U$.
\end{proof}

\subsection{No universal right-angled Artin group}

Kim and Koberda also use a strengthening of Theorem~\ref{thm:main_Kim_Koberda} in the case where $A(\Gamma)$ has dimension at most two ($\Gamma$ has no triangles) to prove the non-existence of a \emph{universal two-dimensional RAAG}, a finitely generated two-dimensional RAAG containing subgroups isomorphic to all finitely generated two-dimensional RAAGs. Using the same technique, we can obtain a stronger statement, which we quote here.

\begin{lem}\label{lem:universal}Let $R\in\mathbb{N}$. There is no finite graph $\Gamma$ with the following property: for any connected finite graph $\Delta$ without induced $r$-gons for $3\leq r\leq R$, one has $\Delta\leq\Gamma\ext_k$. 
\end{lem}

\begin{proof}
Let $\Gamma$ be any finite graph. Then $\Gamma$ is $n$-colourable for some integer $n$ (at most its number of vertices). By Remark~\ref{rem:extension_colour}, $\Gamma\ext$ is also $n$-colourable: pick such a colouring with a set $X$ of colours. Then, define a colouring of $\Gamma\ext_k$ with set of colours $2^X\setminus\{\emptyset\}$ as follows: a vertex $K$ of $\Gamma\ext_k$ represents a family $v_1,\dots, v_m$ of pairwise adjacent vertices of $\Gamma\ext$, which all have distinct colours in $X$. Choose as the colour of $K$ the subset of size $m$ of $X$ containing the colours of the $v_i$. Adjacent vertices of $\Gamma\ext_k$ correspond to distinct sub-cliques of a clique of $\Gamma\ext$, all of whose vertices have distinct colours. Since the sub-cliques are distinct, their colours as subsets of $X$ are distinct. Thus, $\Gamma\ext_k$ is $(2^n-1)$-colourable. By \cite{Erdos_girth}, there exists a connected finite graph $\Delta$ without induced $r$-gons for $3\leq r\leq R$, which is not $(2^n-1)$-colourable. Therefore, $\Delta$ cannot be an induced subgraph of $\Gamma\ext_k$.
\end{proof}

In combination with Theorem~\ref{thm:main_Kim_Koberda}, setting $R=3$, we reach the following corollary.

\begin{cor}
No finitely generated RAAG of any dimension contains subgroups isomorphic to all finitely generated two-dimensional RAAGs. \qed
\end{cor}

\subsection{Extension graphs and clique expansions}

We prove the following results for future use.

\begin{lem}
\label{lem:disjoint_cliques}Let $\Gamma$, $\Delta$ be graphs. Assume that $\Delta$ has a finite number $n\ge 1$ of vertices, and that $\clr(\Delta)\leq \Gamma\ext_k$. Then, there exists a clique expansion $\widetilde\Gamma$ of $\Gamma$ of order $n$ such that $\Delta\leq {\widetilde\Gamma}\ext_k$, and such that the vertices of $\Delta$ correspond to pairwise disjoint cliques of ${\widetilde\Gamma}\ext$.
\end{lem}

\begin{rem}
\label{rem:necessary_expansion}
Note that even under the stronger assumption $\Delta\leq \Gamma\ext_k$, a clique expansion might still be required to obtain disjointness. For example, let $\Gamma$ be a single edge, with vertices $a,b$. We have $\Gamma\ext = \Gamma$ and $\Gamma\ext_k$ is a triangle with vertices $\{a\}, \{b\}, \{a,b\}$. Let $\Delta = \Gamma\ext_k$. Clearly, there is no embedding of $\Delta$ in $\Gamma\ext_k$ with disjoint cliques. However, $\Delta$ embeds in any non-trivial clique expansion of $\Gamma$.
\end{rem}

\begin{proof}
We will first define a clique expansion ${\widetilde\Gamma}$ of $\Gamma$ and a map $\varphi\colon V(\Delta)\to V({\widetilde\Gamma}\ext_k)$, and then show that $\varphi$ extends to an embedding of $\Delta$ as an induced subgraph of ${\widetilde\Gamma}\ext_k$ whose vertices are pairwise disjoint cliques of $\widetilde\Gamma\ext$.

\textbf{Step 1: Definition of ${\widetilde\Gamma}$.} We build ${\widetilde\Gamma}$ from $\Gamma$ by ``blowing-up'' each vertex into an $n$-clique. More precisely:
\begin{itemize}
\item the vertex set of ${\widetilde\Gamma}$ is $V(\Gamma)\times V(\Delta)$; 
\item there is an edge between $(v,x)$ and $(w,y)$ if and only if either $v=w$ and $x\neq y$, or $v$ and $w$ are adjacent. 
\end{itemize}
The clique map ${\widetilde\Gamma}\to \Gamma$ of order $n$ is given by the first coordinate projection. Let $\pi\colon A({\widetilde\Gamma})\to A(\Gamma)$ be the homomorphism induced by this projection. For an arbitrary basepoint $x_0\in V(\Delta)$, there is a section of the clique map given by $v\mapsto (v, x_0)$. This section extends to a homomorphism $s\colon A(\Gamma)\to A({\widetilde\Gamma})$. Clearly on generators, $\pi\circ s = \mathrm{id}_{A(\Gamma)}$.

\textbf{Step 2: Definition of $\varphi\colon V(\Delta)\to V({\widetilde\Gamma}\ext_k)$.} 
Let $f\colon \Delta\to \clr(\Delta)$ be a clique map. For each vertex $K$ of $\Delta$, the vertex $f(K)$ belongs to  $V(\clr(\Delta))$, which in turn is contained in $V(\Gamma\ext_k)$; this enables us to view $f(K)$ as a clique of $\Gamma\ext$. Using Lemma~\ref{lem:extension_cliques}, let $g\in A(\Gamma)$ be any element such that $g^{-1}f(K)g$ is a clique of $\Gamma$. By Remark~\ref{rem:xtypeclique}, $f(K)$ is the extended type of a parabolic subgroup of $A(\Gamma)$, which has uniquely determined type $L\coloneqq g^{-1}f(K)g$, and $g$ is uniquely determined up to right-multiplication by $A(L)\times A(L^\perp)$. By construction, $L\times \{K\}$ is a clique of $\widetilde\Gamma$, hence $s(g)(L\times \{K\})s(g)^{-1}$ is a clique of $\widetilde\Gamma\ext$. Further, if $h$ belongs to $A(L)\times A(L^\perp)\leq A(\Gamma)$, then $h$ is a product of standard generators that all commute with each vertex of $L$. Hence, $s(h)$ is a product of standard generators of $A(\widetilde\Gamma)$ that all commute with each vertex of $L\times \{K\}$ by construction. We have: \[s(gh)(L\times \{K\})s(gh)^{-1} = s(g)(L\times\{K\})s(g)^{-1}.\]

This means that setting \[\varphi(K)\coloneqq s(g)(L\times \{K\})s(g)^{-1} = s(g)(g^{-1}f(K)g\times \{K\})s(g)^{-1}\]
is independent from the choice of $g$. This defines a map $\varphi\colon V(\Delta)\to V(\widetilde\Gamma\ext_k)$.

\textbf{Step 3: The map $\varphi$ is the desired graph embedding.} Now we will prove that $\varphi$ extends to the desired embedding of $\Delta$ as an induced subgraph of ${\widetilde\Gamma}\ext_k$ whose vertices are pairwise disjoint cliques.

\textit{$\drsh$ Injectivity of $\varphi$ and disjoint clique property.} First, note that if $K\neq L$ are vertices of $\Delta$, every vertex of $\varphi(K)$ (resp. $\varphi(L)$) is conjugate in $A({\widetilde\Gamma})$ to a vertex of ${\widetilde\Gamma}$ whose second coordinate is $K$ (resp. $L$). Since distinct vertices of ${\widetilde\Gamma}$ are not conjugate, $\varphi(K)$ and $\varphi(L)$ are disjoint (non-empty) cliques of ${\widetilde\Gamma}\ext$. In particular, $\varphi$ is injective.

\textit{$\drsh$ $\varphi$ preserves adjacency.} Next, if $K$ and $L$ are adjacent vertices of $\Delta$, then $f(K)$ and $f(L)$ are equal or adjacent vertices of $\clr(\Delta)\leq \Gamma\ext_k$: they are contained in a common clique of $\Gamma\ext$. Thus, by Lemma~\ref{lem:extension_cliques}, some element $g^{-1}\in A(\Gamma)$ conjugates both cliques to cliques of $\Gamma$, whose union still induces a clique. By definition of ${\widetilde\Gamma}$, the union of the (disjoint) cliques $g^{-1}f(K)g\times \{K\}$ and $g^{-1}f(L)g\times \{L\}$ of ${\widetilde\Gamma}$ also induces a clique. Conjugating by $s(g)$, the distinct cliques $\varphi(K)$ and $\varphi(L)$ are adjacent in ${\widetilde\Gamma}\ext_k$, proving that $\varphi$ extends to an injective graph morphism.

\textit{$\drsh$ $\varphi$ preserves non-adjacency.} Finally, let $K$ and $L$ be distinct vertices of $\Delta$, and write \[\varphi(K)=s(g)(g^{-1}f(K)g \times \{K\})s(g)^{-1}, \quad \varphi(L) = s(h)(h^{-1}f(L)h \times \{L\})s(h)^{-1}.\] Assume $\varphi(K)$ and $\varphi(L)$ are adjacent vertices of ${\widetilde\Gamma}\ext_k$. Then, $\varphi(K)$ and $\varphi(L)$ are cliques of ${\widetilde\Gamma}\ext$ contained in a common clique. By Lemma~\ref{lem:extension_cliques}, some element $k\in A({\widetilde\Gamma})$ conjugates both cliques to cliques of ${\widetilde\Gamma}$, whose union still induces a clique. But, since no distinct vertices of ${\widetilde\Gamma}$ are conjugate, the only cliques of ${\widetilde\Gamma}$ conjugate to $\varphi(K)$ and $\varphi(L)$ are respectively $g^{-1}f(K)g\times \{K\}$ and $h^{-1}f(L)h\times \{L\}$. This has two consequences:
\begin{itemize}
    \item First, this implies that $ks(g)$ commutes in $A({\widetilde\Gamma})$ with all the vertices of $g^{-1}f(K)g \times \{K\}$, and $ks(h)$ commutes in $A({\widetilde\Gamma})$ with all the vertices of $h^{-1}f(L)h\times \{L\}$. Applying the homomorphism $\pi:A(\widetilde\Gamma)\to A(\Gamma)$ defined in Step~1, this means that $\pi(k)g$ commutes in $A(\Gamma)$ with all the vertices of $g^{-1}f(K)g$, and $\pi(k)h$ commutes in $A(\Gamma)$ with all the vertices of $h^{-1}f(L)h$. Therefore \[g^{-1}f(K)g = \pi(k)g (g^{-1}f(K)g) g^{-1}\pi(k)^{-1} = \pi(k) f(K) \pi(k)^{-1},\] and likewise, $h^{-1}f(L)h = \pi(k) f(L) \pi(k)^{-1}$.
    \item Second, it implies that the union of $g^{-1}f(K)g\times \{K\}$ and $h^{-1}f(L)h\times \{L\}$ induces a clique of ${\widetilde\Gamma}$, i.e.~the vertices of these two cliques are pairwise equal or adjacent. By definition of ${\widetilde\Gamma}$, the vertices of the two cliques $gf(K)g^{-1}$ and $hf(L)h^{-1}$ of $\Gamma$ must be pairwise equal or adjacent, i.e.~the union of $gf(K)g^{-1}$ and $hf(L)h^{-1}$ induces a clique of $\Gamma$. By the first consequence, this can be rephrased by saying that the union of $\pi(k)f(K)\pi(k)^{-1}$ and $\pi(k)f(L)\pi(k)^{-1}$ induces a clique of $\Gamma$. Conjugating by $\pi(k)^{-1}$, the union of $f(K)$ and $f(L)$ induces a clique of $\Gamma\ext$, i.e.~$f(K)$ and $f(L)$ are adjacent or equal vertices of $\Gamma\ext_k$. Since $\clr(\Delta)$ is an induced subgraph of $\Gamma\ext_k$, this means that $f(K)$ and $f(L)$ are adjacent or equal in $\clr(\Delta)$. Since $f$ is a clique map, its edge preimages are cliques, and $K$ and $L$ are adjacent in $\Delta$.
\end{itemize}
Thus, the image of the injective $\varphi$ is an induced subgraph of ${\widetilde\Gamma}\ext_k$, i.e.~$\Delta\leq {\widetilde\Gamma}\ext_k$.
\end{proof}

The following lemma will only be used to prove Theorem~\ref{thm:small_clique_expansion}.

\begin{lem}
\label{lem:can_expand}Let $\Gamma$ be a graph and let $\Delta\leq \Gamma\ext$. For any clique expansion $\widetilde\Delta$ of $\Delta$, there exists a clique expansion $\widetilde\Gamma$ of $\Gamma$ of the same order such that $\widetilde\Delta\leq \widetilde\Gamma\ext$. 
\end{lem}

\begin{proof}
Let $f\colon \widetilde\Delta\to \Delta$ be a clique map of order $o$. Define a graph $\widetilde\Gamma$ and a map $r\colon \widetilde\Gamma\to \Gamma$ such that $r$ is a clique map and every vertex preimage has cardinality $o$ ($\widetilde\Gamma$ is the defining graph of the graph product over $\Gamma$ of a collection of free abelian groups of rank $o$). Let $s\colon A(\Gamma)\to A(\widetilde\Gamma)$ be induced by an arbitrary section of $r$, and let still $r$ denote the induced map $A(\widetilde\Gamma)\to A(\Gamma)$. Clearly on generators, $r\circ s$ is the identity of $A(\Gamma)$. A vertex $x$ of $\Delta\leq \Gamma\ext$ is of the form $gvg^{-1}$ for some $v\in V(\Gamma)$, $g\in A(\Gamma)$. Note that $g$ is uniquely determined up to right-multiplication by an element of $A(\st_\Gamma(v))$. Now, $s(g)r^{-1}(v)s(g)^{-1}$ is a clique of $\widetilde\Gamma\ext$, and for any standard generator $h\in \st_\Gamma(v)$ and $w\in r^{-1}(v)$, we know that $r(s(h))=h$ commutes with $r(w)=v$, hence $s(h)$ commutes with $w$ since $r$ is a clique map. This means that $s(A(\st_\Gamma(v)))$ centralises $r^{-1}(v)$, thus the clique $K_x \coloneqq s(g)r^{-1}(v)s(g)^{-1}$ is independent from the choice of $g$. Note that $r(K_x)=\{gvg^{-1}\} = \{x\}$ since $r\circ s$ is the identity.

By definition of $r$, the clique $K_x$ has cardinality $o$, while the clique $f^{-1}(x)$ has cardinality at most $o$. Choose for each $x\in V(\Delta)$ an injection $\iota_x \colon f^{-1}(x)\to K_x$, %axiom of choice hehe
and define the map
\[\begin{aligned}\iota\colon V(\widetilde\Delta)&\to V(\widetilde\Gamma\ext)\\
v&\mapsto \iota_{f(v)}(v)
\end{aligned}\]
coinciding with each $\iota_x$ on the vertex set of $f^{-1}(x)$. Note that since $r(K_x)=\{x\}$ for all $x\in V(\Delta)$, the map $r\circ \iota$ coincides with $f$. We claim that $\iota$ extends to the desired embedding of $\widetilde\Delta$ as an induced subgraph of $\widetilde\Gamma\ext$.

First make the following observation: if $a\in V(\widetilde\Delta)$ and  $f(a) = g_av_ag_a^{-1}$, with $v_a\in V(\Gamma)$ and $g_a\in A(\Gamma)$ then $\iota(a)$ is of the form $s(g_a)ws(g_a)^{-1}$ for some $w\in r^{-1}(v_a)$. Since $s(v_a)\in r^{-1}(v_a)$ and $r$ is a clique map, $s(v_a)$ and $w$ have the same star, hence they have the same centraliser in $A(\widetilde\Gamma)$. Therefore, $\iota(a)$ and $s(g_a)s(v_a)s(g_a)^{-1} = s(f(a))$ have the same centraliser in $A(\widetilde\Gamma)$.

\textbf{Step 1: $\iota$ is injective.} Let $a,b\in V(\widetilde\Delta)$ such that $\iota(a)=\iota(b)$. Since $r\circ \iota$ coincides with $f$, we have $f(a)=f(b)=x$, and $\iota_{x}(a) = \iota_{x}(b)$. Since $\iota_{x}$ is injective, we get $a=b$.

\textbf{Step 2: $\iota$ preserves adjacency.} Let $a,b$ be two adjacent vertices of $\widetilde\Delta$. Then $f(a)$ and $f(b)$ are adjacent or equal, hence $s(f(a))$ and $s(f(b))$ commute. By the observation above, $\iota(a)$ and $\iota(b)$ commute. Since they are distinct by injectivity, they must be adjacent.

\textbf{Step 3: $\iota$ preserves non-adjacency.} Let $a,b$ be two distinct, non-adjacent vertices of $\widetilde\Delta$. Since $f$ is a clique map, $f(a)$ and $f(b)$ are distinct and non-adjacent: they do not commute in $A(\Gamma)$. Since $r\circ s$ is the identity, $s$ is injective, thus $s(f(a))$ and $s(f(b))$ do not commute in $A(\widetilde\Gamma)$. By the observation above, $\iota(a)$ and $\iota(b)$ do not commute, hence they are distinct and non-adjacent.
\end{proof}

The following lemma will only be used to prove Theorem~\ref{thm:small_clique_expansion}, and in Remark~\ref{rem:finite_unnecessary} in order to extend some of our main results to RAAGs over infinite graphs.

\begin{lem}
\label{lem:clique_graph_expansion}Let $\widetilde \Gamma$ be a clique expansion of countable order of a graph $\Gamma$. Then $\widetilde \Gamma\ext$ is a clique expansion of $\Gamma\ext$.
\end{lem}
\begin{proof}
Let $f_0:\tilde\Gamma\to\Gamma$ be a clique map of countable order. For every $v\in V(\Gamma)$, let $G_v$ be a free abelian group whose rank equals the cardinality of $f_0^{-1}(v)$. Then $A(\widetilde \Gamma)$ is isomorphic to the graph product of $(G_v)$ over $\Gamma$. Note that for each $v\in V(\Gamma)$, all the standard generators of $A(\widetilde\Gamma)$ contained in $G_v$ have the same star in $\widetilde\Gamma$. Let $\widetilde \Gamma^\ast$ be the extension graph of the graph product of $(G_v)$ over $\Gamma$. By Lemma~\ref{lem:extension_graph_product}, $\widetilde\Gamma^\ast$ is isomorphic to $\Gamma\ext$.

\textit{Observation}: the vertices of $\widetilde\Gamma^\ast$ are parabolic subgroups that pairwise intersect trivially. Indeed, assume otherwise. Then, up to conjugating, two distinct such parabolic subgroups of the form $gG_vg^{-1}$, $G_w$ intersect non-trivially, with $v,w\in V(\Gamma)$, $g\in A(\widetilde\Gamma)$. By \cite[Proposition~3.4]{Antolin_Minasyan}, $v=w$ and the intersection is all of $G_w$. Hence, $gG_vg^{-1} = gG_wg^{-1}\subseteq G_w$, and the inclusion is an equality by \cite[Lemma~3.9]{Antolin_Minasyan}, a contradiction. Therefore, each vertex of $\widetilde\Gamma\ext$, seen as an element of $A(\widetilde\Gamma)$, belongs to a unique vertex of $\widetilde\Gamma^\ast$. This defines an $A(\widetilde\Gamma)$-equivariant map $f\colon V(\widetilde\Gamma\ext)\to V(\widetilde\Gamma^\ast)$.

Now, we claim that $f$ extends to a clique map $\widetilde\Gamma\ext\to \widetilde\Gamma^\ast$. First we must prove that $f$ extends to a graph map, i.e.~maps adjacent vertices to equal or adjacent vertices. Since $f$ is $A(\widetilde\Gamma)$-equivariant and every edge of $\widetilde\Gamma\ext$ is conjugate to an edge of $\widetilde\Gamma$ by Lemma~\ref{lem:extension_cliques}, it suffices to check the case of two adjacent vertices $x,y$ of $\widetilde\Gamma$. Let $G_v, G_w$ be the vertex groups of the graph product containing $x$ and $y$ respectively. If $v=w$, then $f$ maps both $x$ and $y$ to the same vertex $G_v$ of $\widetilde\Gamma^\ast$. If $v\neq w$, since all the vertices in $G_v$ (resp. $G_w$) have the same star in $\widetilde\Gamma$ and $x,y$ are adjacent, $G_v$ and $G_w$ must commute. Thus $f$ maps $x$ and $y$ to adjacent vertices in that case. Hence, $f$ extends to a graph map $f\colon \widetilde\Gamma\ext\to \widetilde\Gamma^\ast$. Since the vertex groups of the graph product are all non-trivial, $f$ is surjective on vertices. Since all pairs of elements of two commuting vertex groups commute, $f$ is surjective on edges. Finally, assume that two vertices $gxg^{-1}$,$hyh^{-1}$ of $\widetilde\Gamma\ext$ have the same image under $f$, with $x,y\in V(\widetilde\Gamma)$, $g,h\in A(\widetilde\Gamma)$. Up to conjugating, this image is a vertex of the form $G_v$, $v\in V(\Gamma)$. This means that $gxg^{-1}$, $hyh^{-1}$ are both contained in the abelian parabolic subgroup $G_v\leq A(\widetilde\Gamma)$. In particular, $x,y\in G_v$, and $gxg^{-1} = x$, $hyh^{-1} = y$ by the observation above. These two vertices are adjacent and have the same star in $\widetilde\Gamma$. By Theorem~\ref{thm:centraliser}, $Z_{A(\widetilde\Gamma)}(x) = Z_{A(\widetilde\Gamma)}(y)$. Thus $x$ and $y$ have the same star in $\widetilde\Gamma\ext$. Therefore, $f$ is a clique map $\widetilde\Gamma\ext\to \widetilde\Gamma^\ast$, and $\widetilde\Gamma^\ast\simeq \Gamma\ext$. 
\end{proof}

We now have all the tools to prove Theorem~\ref{thm:small_clique_expansion}.

\begin{proof}[Proof of Theorem~\ref{thm:small_clique_expansion}]
Let $\Gamma$, $\Delta$ be as in the statement. Since $A(\Delta)$ is finitely generated, we can assume that all the vertex groups of the graph product over $\Gamma$ are non-trivial and finitely generated. By Remark~\ref{rem:graph_product_expansion}, there is a clique expansion $\widetilde\Gamma$ of $\Gamma$ of finite order such that $A(\Delta)\leq A(\widetilde\Gamma)$.

By Theorem~\ref{thm:clique_Kim_Koberda}, there is an embedding $\varphi\colon A(\Delta)\hookrightarrow A(\widetilde\Gamma\ext)$ with the following property: for each $v\in V(\Delta)$, there is a finite clique $K_v$ in $\widetilde\Gamma\ext$ such that $\varphi(v)$ belongs to $A(K_v)\leq A(\widetilde\Gamma\ext)$. By definition, $\supp(\varphi(v))$ is a parabolic subgroup of the abelian $A(K_v)$. Thus, we choose $K_v$ minimal, so that $\supp(\varphi(v))=A(K_v)$.

Let $\widetilde\Gamma'$ be the finite subgraph of $\widetilde\Gamma\ext$ induced by the union of the $K_v$ for $v\in V(\Delta)$. Note that $A(\widetilde\Gamma')$ contains all the $\varphi(v)$, hence $\varphi$ is an embedding $A(\Delta)\hookrightarrow A(\widetilde\Gamma')$. Let $\Gamma' \coloneqq \clr(\widetilde\Gamma')$, and let $f\colon \widetilde\Gamma'\to \Gamma'$ be a clique map.
\smallskip

\begin{claim}\label{claim:small_extension} $A(\Delta)$ embeds in a graph product of $\ZZ^{|V(\Delta)|}$ over $\Gamma'$.
\end{claim}
\begin{subproof}
We build the corresponding clique expansion $\bar\Gamma'$ of $\Gamma'$ as in Lemma~\ref{lem:disjoint_cliques}:
\begin{itemize}
    \item the vertex set of $\bar\Gamma'$ is $V(\Gamma')\times V(\Delta)$;
    \item there is an edge between $(a,v)$ and $(b,w)$ if and only if either $a=b$ and $v\neq w$ or $a$ and $b$ are adjacent.
\end{itemize}
The clique map $\bar\Gamma'\to \Gamma'$ of order $|V(\Delta)|$ is given by the first coordinate projection.

We will construct homomorphisms $\varphi_1\colon A(\Delta)\to A(\bar\Gamma')$ and $\varphi_2\colon A(\bar\Gamma')\to A(\widetilde\Gamma')$ such that $\varphi = \varphi_2\circ \varphi_1$. To do so, note that for each $v\in V(\Delta)$, the clique $K_v$ is the join of cliques $K_v\cap f^{-1}(a)$ for $a$ ranging over the vertices of $f(K_v)$. Hence, $A(K_v)$ splits as a direct product

\[A(K_v)\simeq \prod_{a\in V(f(K_v))} A(K_v\cap f^{-1}(a)).\]

Write $\varphi(v)$ as the (commuting) product of $\varphi_a(v)$, $a\in V(f(K_v))$ such that $\varphi_a(v)\in A(K_v\cap f^{-1}(a))$.

Now, define $\varphi_1(v)\in A(\bar\Gamma')$ to be the (commuting) product of $(a,v)$ for all $a\in V(f(K_v))$. The product is commuting because $f(K_v)$ is a clique. If $v,w\in V(\Delta)$ commute, then $\varphi(w)$ is in the centraliser of $\varphi(v)$, which is equal to $A(K_v)\times A(K_v)^\perp$ by Theorem~\ref{thm:centraliser}. Thus, every vertex of $K_w$ commutes with every vertex of $K_v$. The same then holds for every vertex of $f(K_w)$ and every vertex of $f(K_v)$, proving that $\varphi_1(v)$ and $\varphi_1(w)$ commute. Therefore, $\varphi_1$ extends to a homomorphism $A(\Delta)\to A(\bar\Gamma')$.

Next, for $(a,v)\in V(\bar\Gamma')$, if $a\in V(f(K_v))$, define $\varphi_2(a,v)=\varphi_a(v)\in A(\widetilde\Gamma')$, otherwise let $\varphi_2(a,v)$ be the identity. Note that we have $\varphi_2(a,v) \in A(K_v\cap f^{-1}(a))$ in both cases. If $(a,v)$ and $(b,w)$ are commuting vertices of $\bar\Gamma'$, then $a$ and $b$ are commuting vertices of $\Gamma'$ by construction, thus the union of preimages $f^{-1}(a)$ and $f^{-1}(b)$ spans a clique since $f$ is a clique map. This implies that $A(K_v\cap f^{-1}(a))$ and $A(K_w\cap f^{-1}(b))$ commute, hence $\varphi_2(a,v)$ and $\varphi_2(b,w)$ commute. Therefore, $\varphi_2$ extends to a homomorphism $A(\bar\Gamma')\to A(\widetilde\Gamma')$.

Clearly on generators, $\varphi_2\circ \varphi_1 = \varphi$, hence $\varphi_1$ is injective, proving the claim.
\end{subproof}

Now, $\Gamma'$ is a clique-reduced induced subgraph of $\widetilde\Gamma\ext$. By Lemma~\ref{lem:clique_graph_expansion}, the latter is a clique expansion of $\Gamma\ext$. By Lemma~\ref{lem:square_expansion}, $\Gamma'$ is isomorphic to an induced subgraph of $\Gamma\ext$. By Claim~\ref{claim:small_extension}, $A(\Delta)$ embeds into $A(\bar\Gamma')$ for $\bar\Gamma'$ some clique expansion of order $|V(\Delta)|$ of $\Gamma'$. By Lemma~\ref{lem:can_expand}, there is a clique expansion $\bar\Gamma$ of order $|V(\Delta)|$ of $\Gamma$ such that $\bar\Gamma'$ is isomorphic to an induced subgraph of $\bar\Gamma\ext$. Since $\bar\Gamma'$ is finite, $A(\bar\Gamma')$ embeds into $A(\bar\Gamma)$ by Theorem~\ref{thm:main_Kim_Koberda}, hence $A(\Delta)$ embeds into $A(\bar\Gamma)$.
\end{proof}

\subsection{Algorithmic aspects}\label{sec:algorithmic}

We finish this section by establishing the necessary statement to obtain the algorithmic part of Theorem~\ref{thm:main}. This is based on the following result.

\begin{thm}[Casals-Ruiz {\cite{Casals-Ruiz}}]
\label{thm:algorithmic_extension}There is an algorithm which, given two finite graphs $\Gamma$ and $\Delta$, decides whether $\Delta\leq \Gamma\ext$ holds or not.
\end{thm}

We now prove the following corollary.

\begin{cor}
\label{cor:algorithmic_cliques}There is an algorithm which, given two finite graphs $\Gamma$ and $\Delta$, decides whether $\Delta\leq \Gamma\ext_k$ holds or not.
\end{cor}
\begin{proof}
First note that, by definition of the clique graph, for any graphs $\Theta$ and $\Delta$ with $\Delta$ finite, we have $\Delta\leq \Theta_k$ if and only if there exists a family $(K_v)_{v\in V(\Delta)}$ of distinct (finite) cliques of $\Theta$ such that distinct vertices $v,w\in V(\Delta)$ are adjacent if and only if $K_v\cup K_w$ induces a clique of $\Theta$. In that case, we can consider the finite subgraph $\Delta'\leq \Theta$ induced by $\displaystyle \bigcup_{v\in V(\Delta)} K_v$, with a label belonging to $2^{V(\Delta)}\setminus\{\emptyset\}$ on each vertex $x$, namely the set of $v\in V(\Delta)$ such that $x$ is in $K_v$.

From that, we extract a definition: a finite graph $\Delta'$ endowed with a map $f\colon V(\Delta')\to 2^{V(\Delta)}\setminus\{\emptyset\}$ \deffont{represents} $\Delta$ if the following hold:
\begin{itemize}
    \item For every $v\in V(\Delta)$, the subgraph of $\Delta'$ induced by the set of vertices $x$ such that $v\in f(x)$ is a non-empty clique $K_v$.
    \item The cliques $K_v$, $v\in V(\Delta)$ are pairwise distinct.
    \item For distinct vertices $v,w\in V(\Delta)$, the union $K_v\cup K_w$ induces a clique of $\Delta'$ if and only if $v$ and $w$ are adjacent in $\Delta$.
\end{itemize}
By our previous observation, $\Delta\leq \Theta_k$ if and only if there exists a representative $(\Delta',f)$ of $\Delta$ with $\Delta'\leq \Theta$. Note that if $\Theta$ has maximal clique size $n$, representatives of $\Delta$ have at most $n|V(\Delta)|$ vertices. Note also that for any graph $\Gamma$, the maximal clique size of $\Gamma\ext$ is the same as the maximal clique size of $\Gamma$, by Lemma~\ref{lem:extension_cliques}.

We can now write the steps of our algorithm, with input two finite graphs $\Gamma$ and $\Delta$, applying the previous characterisation to $\Theta \coloneqq \Gamma\ext$:
\begin{enumerate}
    \item Compute the maximal clique size $n$ of $\Gamma$.
    \item Make a list $\mathcal{L}$ of all the pairs $(\Delta',f)$ where $\Delta'$ is a graph with at most $n|V(\Delta)|$ vertices and $f$ is a function $V(\Delta')\to 2^{V(\Delta)}\setminus\{\emptyset\}$.
    \item For each entry of $\mathcal{L}$, check if it represents $\Delta$. If yes, keep the entry, if not, discard the entry.
    \item For each kept entry $(\Delta',f)$ of $\mathcal{L}$, run the algorithm from Theorem~\ref{thm:algorithmic_extension} on $\Delta'$ and $\Gamma$, to check if $\Delta'$ is an induced subgraph of $\Gamma\ext$. If the answer is yes for one of the entries, then $\Delta\leq \Gamma\ext_k$. If the answer is no for all entries, then $\Delta\nleq \Gamma\ext_k$. \qedhere
\end{enumerate}
\end{proof}

\section{Background on measurable embeddings}\label{section:background_me}

In this section, we review measured groupoids and their relationship to measurable embeddings. General references for the language of measured groupoids include \cite{ADR,kidaintro} or \cite[Section~3]{guirardelhorbez} and the references therein.

\subsection{Measured groupoids and cocycles}

A \deffont{Borel groupoid} over a standard Borel space $X$ is a standard Borel space $\mclG$ equipped with 
\begin{itemize}
    \item two Borel maps $s:\mclG\to X$ and $r:\mclG\to X$ (the \deffont{source} and \deffont{range} maps);
    \item an associative Borel \deffont{composition map} 
    \[\begin{array}{cccc}
    \mclG^{(2)}&\to &\mclG\\
    (\underline{g}_1,\underline{g}_2) &\mapsto & \underline{g}_1\underline{g}_2
    \end{array}\]
    defined on $\mclG^{
(2)} := \{(\underline{g}_1,\underline{g}_2)\in\mclG\times\mclG \mid s(\underline{g}_1)= r(\underline{g}_2)\}$, such that $s(\underline{g}_1\underline{g}_2)=s(\underline{g}_2)$ and $r(\underline{g}_1\underline{g}_2)=r(\underline{g}_1)$ for all $(\underline{g}_1,\underline{g}_2)\in\mclG^{(2)}$;
\item a Borel \deffont{unit map} \[\begin{array}{cccc}
    X&\to &\mclG\\
    x &\mapsto & \underline{e}_x
    \end{array}\] such that $s(\underline{e}_x)=r(\underline{e}_x)=x$, and $\underline{g}\underline{e}_x=\underline{g}$ whenever $s(\underline{g})=x$, and $\underline{e}_x\underline{g}=\underline{g}$ whenever $r(\underline{g})=\underline{e}_x$;
\item a Borel \deffont{inverse map} \[\begin{array}{cccc}
    \mclG&\to &\mclG\\
    \underline{g} &\mapsto & \underline{g}^{-1}
    \end{array}\]
    such that $s(\underline{g}^{-1})=r(\underline{g})$ and $r(\underline{g}^{-1})=s(\underline{g})$ and $\underline{g}^{-1}\underline{g}=\underline{e}_{s(\underline{g})}$ and $\underline{gg}^{-1}=\underline{e}_{r(\underline{g})}$.
\end{itemize}
It is \deffont{discrete} if preimages of the range and source maps are countable.

\medskip

A theorem of Lusin--Novikov (see e.g.\ \cite[Theorem~18.10]{kechris}) ensures that every discrete Borel groupoid is covered by countably many \deffont{bisections}, i.e.\ Borel subsets $B\subseteq\mclG$ such that $s_{|B}$ and $r_{|B}$ are isomorphisms to Borel subsets $s(B)$ and $r(B)$ of $X$. A \deffont{measured groupoid} is a discrete Borel groupoid over a standard measure space $X$, together with a Borel measure $\mu$ on $X$ such that for every bisection $B\subseteq\mclG$, one has $\mu(s(B))=0$ if and only if $\mu(r(B))=0$.

Given a measured groupoid $\mclG$ over $(X,\mu)$ and a Borel subset $U\subseteq X$, the restriction \[\mathcal{G}_{|U}:=\{\underline{g}\in\mclG \mid s(\underline{g}),r(\underline{g})\in U\}\] is naturally a measured groupoid over $(U,\mu_{|U})$. A measured groupoid $\mclG$ over $(X,\mu)$ is of \deffont{infinite type} if for any Borel subset $U\subseteq X$ of positive measure and for almost every $x\in U$, there are infinitely many $\underline{g}\in\mclG_{|U}$ such that $s(\underline{g})=x$.

Given a measured groupoid $\mclG$ and a countable group $G$, a measurable map $\rho:\mclG\to G$ is a \deffont{cocycle} if $\rho(\underline{g}_1\underline{g}_2)=\rho(\underline{g}_1)\rho(\underline{g}_2)$ for all $(\underline{g}_1,\underline{g}_2)\in\mclG^{(2)}$. It has \deffont{trivial kernel} if every element $\underline{g}\in\mclG$ such that $\rho(\underline{g})=1$ is a unit, i.e.\ of the form $\underline{e}_x$ for some $x\in X$. Following \cite[Definition~3.21]{guirardelhorbez}, we say that a cocycle $\rho:\mclG\to G$ is of \deffont{action type} if 
\begin{itemize}
    \item it has trivial kernel, and
    \item for every infinite subgroup $H\subseteq G$, the groupoid $\rho^{-1}(H)$ is of infinite type.
\end{itemize}
Note that if $\rho$ has trivial kernel (resp.\ is of action type), then for every positive measure Borel subset $U\subseteq X$, the restriction of $\rho$ to $\mclG_{|U}$ has trivial kernel (resp.\ is of action type).

\begin{ex}\label{cocycles from pmp actions are action-type}
    Let $G$ be a countable group acting by measure-preserving Borel automorphisms on a standard measure space $X$. Then $G\times X$ is naturally a measured groupoid over $X$, by letting $s(g,x)=x$ and $r(g,x)=gx$, defining the composition map by $(g,hx)(h,x)=(gh,x)$, the unit map by $\underline{e}_x=(e,x)$, and the inverse map by $(g,x)^{-1}=(g^{-1},gx)$. This measured groupoid is denoted by $G\ltimes X$. It comes with a natural cocycle $\rho$ towards $G$, by letting $\rho(g,x)=g$. This cocycle always has trivial kernel; when $X$ has finite measure, it is of action type by \cite[Proposition~2.26]{kidaintro}.
\end{ex}

\subsection{Amenability and normalisation}

There is a notion of amenability of a measured groupoid which generalises Zimmer's notion of amenability of a group action \cite[Definition~1.4]{Zimmer}, see e.g.\ \cite[Definition~3.36]{guirardelhorbez}. There is also a notion of normalisation of a measured subgroupoid by another, following \cite{FSZ,kidamemoire}, see e.g.\ \cite[Definition~3.33]{guirardelhorbez}. We refer the reader to the given references for the definitions, and only gather here some facts we will need in the present paper (they will be used exclusively in the proof of Proposition~\ref{prop:parabolic_support}).

\begin{facts}
\label{facts}Let $\mclG$ be a measured groupoid over a standard probability space $X$, and let $\rho:\mclG\to G$ be a cocycle.
\begin{enumerate}
    \item\label{fact-1} If $\rho$ has trivial kernel and $G$ is amenable, then $\mclG$ is amenable (see e.g. \cite[Corollary~3.42]{guirardelhorbez}).
    \item\label{fact-2} If $\rho$ is of action type and $G$ contains a non-abelian free group, then $\mclG$ is non-amenable (see \cite[Lemma~3.20]{kidaannals} or \cite[Lemma~3.43]{guirardelhorbez}).
    \item\label{fact-3} If $K,H\subseteq G$ are subgroups such that $K$ is normalised by $H$, then $\rho^{-1}(K)$ is normalised by $\rho^{-1}(H)$ (see e.g.\ \cite[Lemma~3.35]{guirardelhorbez}).
\end{enumerate}
\end{facts}

Amenability and normalisation are stable under restriction to a Borel subset of $X$.

\subsection{Measure equivalence and measurable embeddings}\label{section meas embeddings}

The following definition is at the heart of the present work.

\begin{defn}
    Let $G,H$ be two countable groups. We say that $H$ \deffont{measurably embeds} into $G$ if there exists a standard Borel space $(\Sigma,m)$ with a measure-preserving action of $G\times H$ such that
    \begin{itemize}
        \item the $G$-action on $\Sigma$ is free and admits a measurable fundamental domain of finite positive measure,
        \item the $H$-action on $\Sigma$ is free and admits a measurable fundamental domain.
    \end{itemize}
    The groups $G$ and $H$ are 
    \begin{itemize}
    \item \deffont{measure equivalent} if there exists $(\Sigma,m)$ as above where additionally the $H$-action has a mesurable fundamental domain of finite measure;
    \item \deffont{orbit equivalent} if there exists $(\Sigma,m)$ as above where additionally $G$ and $H$ have a common measurable fundamental domain of finite measure. 
    \end{itemize}
\end{defn}

\begin{rem}\label{rem:oe-me}
Orbit equivalence of $G$ and $H$ is also characterised by the existence of free measure-preserving actions of $G$ and $H$ on a standard probability space $X$ such that for almost every $x\in X$, one has $H\cdot x=G\cdot x$, see \cite{Furman_oe} and \cite[Section~2.1]{gaboriau_examples}.

Clearly from the definitions, every orbit equivalence is a measure equivalence, and every measure equivalence is a measurable embedding.

Moreover, assume there exists a standard probability space $X$ equipped with two measure-preserving actions of $G$ and $H$ such that for almost every $x\in X$, one has $H\cdot x\subseteq G\cdot x$. Then $H$ measurably embeds into $G$ by \cite[Remark~2.36]{DKLMT}.
\end{rem}

\begin{rem}\label{transitivité du plongement ME}
Measure equivalence and orbit equivalence are equivalence relations on the set of all countable groups, see \cite[Section~2]{Furman_me}.

Measurable embedding is a pre-order. In particular, if $G_1$ measurably embeds into $G_2$ and $G_2$ into $G_3$, then $G_1$ measurably embeds into $G_3$. This is shown for instance in \cite[Proposition~2.9]{DKLMT}.
\end{rem}

\begin{rem}\label{alg embedding is meas embedding}
Let $G,H$ be two countable groups. If $H$ embeds as a subgroup in $G$, via an injective homomorphism $f:H\to G$, then $H$ measurably embeds into $G$. Indeed in this case, one can take $\Sigma=G$ with the counting measure, and consider the action of $H\times G$ on $\Sigma$ given by $(h,g)\cdot g'=f(h)g'g^{-1}$. The action of $G$ has a fundamental domain reduced to a single point.
\end{rem}

The use of measured groupoids in the study of measurable embeddings comes from the following lemma.

\begin{lem}[{see \cite[Lemma~2.18]{gurieva}}]
\label{lem:me-embedding}Let $G,H$ be two countable groups, and assume that $H$ measurably embeds into $G$. Then there exists a measured groupoid $\mclG$ over a standard probability space, equipped with an cocycle of action type $\rho_H:\mclG\to H$, and a cocycle with trivial kernel $\rho_G:\mclG\to G$.
\end{lem}

We briefly explain how to construct the groupoid $\mclG$ from the measurable embedding, and refer to the proof of \cite[Lemma~2.18]{gurieva} for details. Fix a measurable embedding from $H$ to $G$, given by a coupling $(\Sigma,m)$, and take respective fundamental domains $D_H$ and $D_G$ whose intersection $X$ has positive measure. Identifying $D_H\simeq \Sigma/H$ and $D_G\simeq \Sigma/G$ and using that the actions of $G$ and $H$ commute, we get two measure-preserving actions $G\curvearrowright D_H$ and $H\curvearrowright D_G$. One checks that in restriction to $X$, we have an isomorphism of groupoids 
\[\mclG = (H\ltimes D_G)_{|X} \simeq (G\ltimes D_H)_{|X}~,\]
so $\mclG$ comes equipped with two cocycles that are restrictions of cocycles arising from group actions as in Example~\ref{cocycles from pmp actions are action-type}:
\[\rho_H \colon (H\ltimes D_G)_{|X} \to H, \quad \rho_G \colon (G\ltimes D_H)_{|X} \to G.\]
The cocycle $\rho_G$ has trivial kernel and, since $D_G$ has finite measure, the cocycle $\rho_H$ is of action type.

\subsection{Orbit equivalence and right-angled Artin groups}

In the context of RAAGs, the following theorem was proved in \cite{Horbez_Huang} by combining the Ornstein--Weiss theorem stating that all countably infinite amenable groups are orbit equivalent \cite{OW} with an argument of Gaboriau \cite{gaboriau_examples}.

\begin{thm}[{\cite[Corollary~5]{Horbez_Huang}}]
\label{thm:amenable_graph_product}Let $\Gamma$ be a countable graph, and $(G_v)_{v\in V(\Gamma)}$ be a family of countably infinite amenable groups. Then, the graph product of $(G_v)$ over $\Gamma$ is orbit equivalent to $A(\Gamma)$.
\end{thm}
The following corollary is then a direct consequence of Remark~\ref{rem:graph_product_expansion}.
\begin{cor}
\label{cor:ME_clique_reduction}
For every countable graph $\Gamma$, the RAAGs $A(\Gamma)$ and $A(\clr(\Gamma))$ are orbit equivalent.
\end{cor}
\medskip

\begin{rem}
The statement and proof given in \cite{Horbez_Huang} assume that the graph $\Gamma$ is finite. Another presentation of the argument, which readily extends to the case of a countable graph, is given in \cite[Section~11]{Escalier_Horbez} under the heading ``Construction of the pairing''.
\end{rem}
\medskip

\section{From measurable embeddings of RAAGs to graph embeddings}\label{section:me_to_graph}

The goal of this section is to establish the following theorem, which is at the core of the implication (\ref{main-1})$\Rightarrow$(\ref{main-4}) from Theorem~\ref{thm:main}.

\begin{thm}
\label{th meas embedding to graph}Let $\Delta$ be a clique-reduced finite graph and $\Gamma$ a finite graph without induced squares. Assume that $A(\Delta)$ measurably embeds into $A(\Gamma)$. Then $\Delta\leq\Gamma\ext_k$.
\end{thm}
\medskip

Theorem~\ref{th meas embedding to graph} follows directly from the combination of Propositions~\ref{prop:parabolic_support} and~\ref{prop:witness-to-graph} below.

\subsection{Parabolic witness maps}

Given a measurable embedding from $A(\Delta)$ to $A(\Gamma)$, the first step in our proof of Theorem~\ref{th meas embedding to graph} is to obtain a map at the level of parabolic subgroups, as made precise in the following definition.
\medskip

\begin{defn}
\label{defn:witness}Given two  graphs $\Delta,\Gamma$, a \deffont{parabolic witness map} from $\Delta$ to $\Gamma$ is a map 
\[\begin{aligned}
V(\Delta) &\to \{\text{non-trivial parabolic subgroups of }A(\Gamma)\}\\
v&\mapsto P_v
\end{aligned}\]
such that the following properties hold for all distinct $u, v \in V(\Delta)$:
\begin{enumerate}[label={(\alph*)}]
    \item\label{witness-a} If $u$ and $v$ are adjacent in $\Delta$, then $P_u$ and $P_v$ normalise each other.
    \item\label{witness-b} If $P_u$ and $P_v$ normalise each other, and one of them is abelian, then $u$ and $v$ are adjacent in $\Delta$.
    \item\label{witness-c} If $P_v^\perp$ is abelian, then $\{v\}^\perp$ is a clique in $\Delta$.
\end{enumerate}
\end{defn}

\begin{rem}\label{rem:algebraic_embeddings}
Observe that, if $A(\Delta)$ embeds as a subgroup into $A(\Gamma)$ via an injective homomorphism $f$, then  setting $P_v\coloneqq \supp(f(v))$ for each $v\in V(\Delta)$, the map $v\mapsto P_v$ is a parabolic witness map.

Indeed, if $u,v$ are distinct vertices of $\Delta$ which commute, then $f(u)$ and $f(v)$ commute. By Theorem~\ref{thm:centraliser}, $f(u)\in  P_v \times P_v^\perp$, hence $P_u\subseteq P_v\times P_v^\perp$, proving that $P_u$ normalises $P_v$ (and vice-versa), thus \ref{witness-a} holds.

Next, if $P_u$ and $P_v$ normalise each other and, say, $P_u=\supp(f(u))$ is abelian, we know that $N_{A(\Gamma)} (P_u) = \supp(f(u))\times \supp(f(u))^\perp=Z_{A(\Gamma)} (f(u))$ by Theorem~\ref{thm:centraliser}. Since $P_v$ normalises $P_u$, we have $\supp(f(v))\leq Z_{A(\Gamma)} (f(u))$, thus $f(v) \in Z_{A(\Gamma)} (f(u))$, proving that $f(u)$ and $f(v)$ commute. Since $f$ is an embedding, $u$ and $v$ commute, and \ref{witness-b} holds.

Finally, if $P_v^\perp$ is abelian, by Theorem~\ref{thm:centraliser} again, $Z_{A(\Gamma)} (f(v))$ is a direct product of abelian groups, thus it is itself abelian. Now $f$ embeds the centraliser $A(\st_\Delta(v))$ of $v$ into the abelian centraliser of $f(v)$, hence $A(\st_\Delta(v))$ must itself be abelian: $\st_\Delta(v)$ is a clique, and \ref{witness-c} holds.
\end{rem}
\medskip

The following definition will be central in the section.
\smallskip

\begin{defn}
\label{de:hat2}A vertex $v$ of a graph $\Gamma$ is a \deffont{hat} if $\st_\Gamma(v)$ is a clique.
\end{defn}

The following lemma says that when $\Gamma$ is square free, any parabolic witness map from $\Delta$ to $\Gamma$ has the extra property that its values on non-hat vertices are abelian parabolic subgroups. The reader is referred to Definition~\ref{def:xtype} for the extended type $\xtype(P)$ of a parabolic subgroup $P$.

\begin{lem}
\label{lemma nonhat has clique support}Let $\Gamma,\Delta$ be finite graphs, and let \[\begin{aligned}
V(\Delta) &\to \{\text{non-trivial parabolic subgroups of }A(\Gamma)\}\\
v&\mapsto P_v
\end{aligned}\] be a parabolic witness map from $\Delta$ to $\Gamma$. Assume that $\Gamma$ has no induced square and $u\in V(\Delta)$ is not a hat. Then, $P_u$ is abelian (i.e.~$\xtype(P_u)$ is a clique).
\end{lem}

\begin{proof}
By Item~\ref{witness-c} from the definition of a parabolic witness map (Definition~\ref{defn:witness}),  we know that $P_u^\perp$ is not abelian, thus $\xtype(P_u^\perp)$ contains two non-adjacent vertices $x, x'$ of $\Gamma\ext$. Assuming that $\xtype(P_u)$ contains two non-adjacent vertices $y, y'$ immediately contradicts Corollary~\ref{cor:square_F2xF2}, since $x, y, x', y'$ then forms an induced square in $\Gamma\ext$. Thus, $\xtype(P_u)$ is a clique.
\end{proof}

\subsection{From measurable embeddings to parabolic witness maps}
\medskip

Our next proposition extends the contents of Remark~\ref{rem:algebraic_embeddings} from algebraic to measurable embeddings. The key tool is a notion of parabolic support for a measured groupoid endowed with a cocycle towards a RAAG, coming from the work of Huang and the third-named author \cite[Section~3.3]{Horbez_Huang}.

\begin{prop}
\label{prop:parabolic_support}Let $\Delta$, $\Gamma$ be finite graphs. Assume that $A(\Delta)$ measurably embeds into $A(\Gamma)$. Then there exists a parabolic witness map from $\Delta$ to $\Gamma$.
\end{prop}

\begin{proof}
Since $A(\Delta)$ measurably embeds into $A(\Gamma)$, Lemma~\ref{lem:me-embedding} gives the existence of a measured groupoid $\mclG$ over a standard probability space $X$, with two cocycles 
\[\rhoL\colon\mclG\to A(\Delta) \quad \text{~and~} \quad \rhoG\colon \mclG \to \AG,\] such that $\rhoL$ is of action type and $\rhoG$ has trivial kernel. 

We now apply \cite[Lemma~3.7]{Horbez_Huang}, with $\mathbb{P}$ the set of all parabolic subgroups of $A(\Gamma)$, to define the parabolic support (with respect to $\rhoG$) of each of the finitely many subgroupoids $\rhoL^{-1}(\gen v)$, with $v$ varying in the finite set $V(\Delta)$. More precisely, \cite[Lemma~3.7]{Horbez_Huang} gives the existence a positive measure Borel subset $U$ of $X$, such that for every $v\in V(\Delta)$, there exists a parabolic subgroup $P_v$ of $A(\Gamma)$ with the following properties: 
\begin{enumerate}
    \item\label{support1} one has $\rhoG(\rhoL^{-1}(\gen v)_{|U})\subseteq P_v$;
    \item\label{support2} for every parabolic subgroup $Q$ of $\AG$, and every Borel subset $V\subseteq U$ of positive measure, if $\rhoG(\rhoL^{-1}(\gen v)_{|V})\subseteq Q$ then $P_v\subseteq Q$.
\end{enumerate}
For every $v\in V(\Delta)$, we now set \[\mclA_v := \rhoL^{-1}(\gen v)_{|U}.\]
Let us show that $v\mapsto P_v$ is a parabolic witness map.

We first check that $P_v$ is non-trivial for any vertex $v$ of $\Delta$. The subgroupoid $\mclA_v$ is of infinite type because the resctriction of $\rho_{\Delta}$ to ${\mclG_{|U}}$ is of action type. The cocycle $\rho_{\Gamma}$ having trivial kernel, $\rho_{\Gamma}(\mclA_v)$ is an infinite subset of $P_v$. This forces $P_v$ to be infinite.

To prove Item~\ref{witness-a}, fix two adjacent vertices $u$ and $v$ in $\Delta$. Since $\langle u\rangle$ and $\langle v\rangle$ commute, the subgroupoids $\mclA_u$ and $\mclA_v$ normalise each other (Facts~\ref{facts}(\ref{fact-3})). By \cite[Lemma~3.8]{Horbez_Huang}, up to replacing $U$ by a conull Borel subset, we have $\rhoG(\mclA_u)\subseteq P_v\times P_v^\perp$ and $\rhoG(\mclA_v)\subseteq P_u\times P_u^\perp$. By minimality of the parabolic support (Property~(\ref{support2}) above), $P_u\subseteq P_v\times P_v^\perp$ and $P_v\subseteq P_u\times P_u^\perp$, hence $P_u$ and $P_v$ normalise each other.

Let us prove Item~\ref{witness-b}. Take two distinct vertices $u$ and $v$ of $\Delta$. Assume that $P_u$ is abelian and $P_u$ and $P_v$ normalise each other. Then, 
$$\rhoG(\mclA_u)\subseteq P_u\subseteq Z(P_u\times P_u^\perp)$$
and
$$\rhoG(\mclA_v)\subseteq P_v\subseteq P_u\times P_u^\perp.$$
We aim to prove that $u$ and $v$ are adjacent in $\Delta$, so let us assume towards a contradiction that they are not. In particular $\langle u,v\rangle=\langle u\rangle\ast\langle v\rangle$ is a non-abelian free group. The groupoid $\langle\mclA_u,\mclA_v\rangle$ generated by $\mclA_u$ and $\mclA_v$ comes equipped with a cocycle $\rhoL$ towards $\langle u,v\rangle$ with trivial kernel, and a cocycle $\rhoG$ towards $P_u\times P_u^\perp$ with trivial kernel. We now apply \cite[Lemma~2.26]{gurieva} to the groupoid $\langle\mclA_u,\mclA_v\rangle$ with $\w$ and $\rho$ the respective restrictions of $\rhoL$ and $\rhoG$, with $A=\gen{u}$ and $K=\gen{v}$, with $\mclA=\mclA_u$ and $\mathcal{K}=\mclA_v$, and with $G=P_u\times P_u^\perp$. It follows that $\rhoG(\mclA_u)$ cannot be contained in $Z(P_u\times P_u^\perp)$, a contradiction.

Finally, to prove Item~\ref{witness-c}, fix a vertex $v$ of $\Delta$ such that $P_v^\perp$ is abelian. Let
\[\mclH_v := \rhoL^{-1}\left(A(\lk_\Delta(v))\right)_{|U}.\]
The group $\langle v\rangle$ is amenable and the cocycle $\rho_\Delta$ has trivial kernel, hence the subgroupoid $\mclA_v$ is amenable (Facts~\ref{facts}(\ref{fact-1})). As $\langle v\rangle$ is normalised by $A(\lk_\Delta(v))$, the subgroupoid $\mclA_v$ is normalised by $\mclH_v$ (Facts~\ref{facts}(\ref{fact-3})). By \cite[Lemma~3.10]{Horbez_Huang}, the amenability of $P_v^\perp$ forces $\mclH_v$ to be amenable.  If $A(\lk_\Delta (v))$ were non-abelian, it would contain a non-abelian free group; as $\rhoL$ is of action type, this would imply that $\mclH_v$ is non-amenable (Facts~\ref{facts}(\ref{fact-2})), a contradiction. Thus, $A(\lk_\Delta(v))$ is abelian, i.e.\ $\lk_\Delta(v)$ is a clique.
\end{proof}

\subsection{From parabolic witness maps to graph embeddings: a simple case}

In this section, under extra assumptions on the graphs $\Delta$ and $\Gamma$, given a parabolic witness map from $\Delta$ to $\Gamma$, we build an embedding $\Delta\leq\Gamma\ext$, without having to pass to the clique graph $\Gamma\ext_k$. The content of this section is not needed for the proof of Theorem~\ref{th meas embedding to graph} and for our main result, Theorem~\ref{thm:main}. It serves as a warm-up for the argument in the next section, and will be used in the proof of Theorem~\ref{thm:meas-algebraic}, generalising Theorem~\ref{thm:intro-meas-algebraic} from the introduction.

\begin{prop}
\label{prop:witness-to-graph-warmup}Let $\Delta$ be a finite clique-reduced graph where no link of a vertex is a clique with at least two vertices. Let $\Gamma$ be a graph without induced triangles or squares. Assume that there exists a parabolic witness map from $\Delta$ to $\Gamma$. Then $\Delta\leq \Gamma\ext$.
\end{prop}

\begin{rem}
When the conclusion of Proposition~\ref{prop:witness-to-graph-warmup} holds, Theorem~\ref{thm:main_Kim_Koberda} provides an algebraic embedding of $A(\Delta)$ into $A(\Gamma)$ right away.

We also mention that under the extra assumption that $\Delta$ has no vertices of valence $0$ or $1$, the proof simplifies further, and Step~2 is not needed.
\end{rem}

\begin{proof}
Note that if $A(\Gamma)=\{1\}$, then unless $A(\Delta)=\{1\}$ there cannot be any parabolic witness map from $\Delta$ to $\Gamma$. If $A(\Gamma)$ is infinite abelian, then Item~\ref{witness-b} from Definition~\ref{defn:witness} ensures that $\Delta$ is a clique; being clique reduced it must be empty or a vertex and the conclusion holds. From now on we assume that $A(\Gamma)$ is non-abelian.

Let $\widehat\Delta\leq\Delta$ be the subgraph of $\Delta$ induced by vertices of valence $0$ or $1$. Our assumption on $\Delta$ ensures that these are exactly the hat vertices of $\Delta$ in the sense of Definition~\ref{de:hat2}. Denote by $v\mapsto P_v$ a parabolic witness map from $\Delta$ to $\Gamma$.

\medskip

\textbf{Step 0:} We show that for every $v\in V(\Delta)\setminus V(\widehat{\Delta})$, the parabolic subgroup $P_v$ is cyclic, and $P_v^\perp$ is a non-abelian free group. If in addition $v$ is not central in $A(\Delta)$, then $P_v$ is not central in $A(\Gamma)$.

Indeed, let $v\in V(\Delta)\setminus V(\widehat\Delta)$. Since $v$ is not a hat, Item \ref{witness-c} from the definition of a parabolic witness map (Definition~\ref{defn:witness}) shows that $P_v^\perp$ is non-abelian. Then, the graph $\Gamma$ having no triangles nor squares forces $P_v$ to be a cyclic parabolic subgroup, and $P_v^\perp$ to be a non-abelian free group. Finally, if $v$ is not central in $A(\Delta)$, we can find $u\in V(\Delta)$ that is distinct from and not adjacent to $v$. It then follows from Item~\ref{witness-b} from Definition~\ref{defn:witness} that $P_v$ does not commute to $P_u$, in particular $P_v$ is not central in $A(\Gamma)$.

\medskip

\textbf{Step 1:} By Step~0, the map
\[\theta\colon \begin{array}{ccc}
V(\Delta)\setminus V(\widehat\Delta)&\to &V(\Gamma\ext)\\
v&\mapsto &\xtype(P_v)
\end{array}\]
is well defined. We claim that it extends to an embedding $\Delta\setminus\widehat\Delta\leq\Gamma\ext$ (here we write $\Delta\setminus\widehat{\Delta}$ as a shortcut for the subgraph of $\Delta$ induced by vertices not in $\widehat{\Delta}$).

Indeed, Items~\ref{witness-a} and~\ref{witness-b} from Definition~\ref{defn:witness} mean that for any two vertices $u,v\in V(\Delta)\setminus V(\widehat \Delta)$, $u$ and $v$ are adjacent if and only if $\xtype(P_u)$ and $\xtype(P_v)$ are adjacent. 

We thus only need to show that the map sending $v\in V(\Delta)$ to $\xtype(P_v)\in V(\G\ext)$ is injective. It is indeed the case: for any two distinct $u,v\in V(\Delta)$, there exists $w\in V(\Delta)\setminus\{u,v\}$ adjacent to one and not to the other because $\Delta$ is clique reduced. Then, $\xtype (P_w)$ is adjacent to exactly one of $\xtype(P_u)$ and $\xtype(P_v)$, which means in particular that $\xtype(P_u)$ and $\xtype(P_v)$ are distinct vertices in $\Gamma\ext$. 

\medskip

\textbf{Step 2:} We extend the above embedding to an embedding $\Delta\leq\Gamma\ext$. 

To this end, consider a chain of induced subgraphs
\[\Delta\setminus\widehat\Delta=\Delta_0\subseteq\Delta_1\subseteq\dots\subseteq\Delta_\ell=\Delta,\] such that
for every $i\in\{1,\dots,\ell\}$, the difference $V(\Delta_i)\setminus V(\Delta_{i-1})$ consists of exactly one link-equivalence class $[v_i]_{\lk}$ of vertices of $\Delta$ (of valence $0$ or $1$). In Step~1 we have defined $\theta$ on $\Delta_0$. We now inductively extend it to a map $\theta\colon V(\Delta_i)\to V(\Gamma\ext)$ that induces an embedding $\Delta_i\leq\Gamma\ext$, and such that for every $w\in V(\Delta_i)\setminus V(\Delta_0)$, the element $\theta(w)$ is not central.

Let $i\in\{1,\dots,\ell\}$, and write $[v_i]_{\lk}=\{v_{i,1},\dots,v_{i,n}\}$. Let $Q$ be the subgroup of $A(\Gamma)$ generated by all $P_w$ with $w\in \lk_\Delta(v_i)$. We claim that $Q$ is finitely generated abelian, $Q^\perp$ is non-abelian, and no $\theta(u)$ with $u\in V(\Delta_{i-1})\setminus\lk_\Delta(v_i)$ is contained in $Q\times Z(Q^\perp)$. Indeed:
\begin{itemize}
\item If $v_i$ has valence $0$, then $Q=\{1\}$, so $Q^\perp=A(\Gamma)$ is non-abelian. Yet no $\theta(u)$ with $u\in V(\Delta_{i-1})$ is contained in $Z(A(\Gamma))$: this follows from Step~0 if $u\in V(\Delta_0)$ (noting that $u$ is not central in $A(\Delta)$ because it does not commute with $v_i$), and from our induction hypothesis otherwise.
\item If $v_i$ has valence $1$, letting $w$ be the unique vertex in $\lk_\Delta(v_i)$, we have $Q=P_w$. Since $\Delta$ is clique reduced, no connected component of $\Delta$ is an edge, so $w\in V(\Delta_0)$. In particular, Step~0 ensures that $Q$ is a cyclic parabolic subgroup, and $Q^\perp$ is a non-abelian free group, so $Q\times Z(Q^\perp)=Q=P_w$. It follows from the injectivity of $\theta$ on $V(\Delta_{i-1})$ that  no $\theta(u)$ with $u\in V(\Delta_{i-1})\setminus\lk_\Delta(v_i)$ is contained in the cyclic group $Q\times Z(Q^\perp)$.
\end{itemize}
Lemma~\ref{lemma:xtype} thus enables to find $n$ vertices $x_1,\dots,x_n\in V(\xtype(Q^\perp)\setminus\xtype(Z(Q^\perp)))$ that are pairwise non-adjacent, and distinct from and not adjacent to any of the vertices $\theta(u)$ with $u\in V(\Delta_{i-1})\setminus\lk_\Delta(v_i)$. Setting $\theta(v_{i,j}):=x_j$ completes the induction. 
\end{proof}

\subsection{From parabolic witness maps to graph embeddings: the general case}

In order to complete our proof of Theorem~\ref{th meas embedding to graph}, we are left with showing the following proposition.

\begin{prop}
\label{prop:witness-to-graph}Let $\Delta$ be a finite clique-reduced graph and $\Gamma$ a graph without any induced square. Assume that there exists a parabolic witness map from $\Delta$ to $\Gamma$. Then $\Delta\leq\Gamma\ext_k$. 
\end{prop}

\begin{proof}
Denote by $\widehat \Delta$ the subgraph of $\Delta$ induced by hat vertices. 

\medbreak

\textbf{Step 1: The map on non-hat vertices.} Lemma~\ref{lemma nonhat has clique support} enables us to define a map
\[\theta \colon\begin{array}{ccc}
     V(\Delta)\setminus V(\widehat \Delta)  & \to & V(\Gamma\ext_k) \\
     v&\mapsto & \xtype(P_v)
\end{array}\]
We will now prove that $\theta$ has the following property:
\begin{enumerate}[label = ($\bigstar$)$_0$]
\item \label{condition1} For any two vertices $u,v\in V(\Delta)\setminus V(\widehat\Delta)$, the union $\theta(u)\cup\theta(v)$ spans a clique if and only if $u$ and $v$ are adjacent or equal.
\end{enumerate}

First, let $u,v\in V(\Delta)\setminus V(\widehat\Delta)$ be distinct vertices such that $\xtype(P_u)\cup \xtype(P_v)$ spans a clique in $\G\ext_k$. Then $P_u$ and $P_v$ are abelian and normalise each other. By Item~\ref{witness-b} of Definition~\ref{defn:witness}, $u$ and $v$ are adjacent. 

Conversely, let $u,v\in V(\Delta)\setminus V(\widehat\Delta)$ be two adjacent vertices. By Item~\ref{witness-a} of Definition~\ref{defn:witness}, $P_u\subseteq P_v\times P_v^\perp.$
But $P_u$ is abelian, so $P_u\cap P_v^\perp$ is abelian. Thus $\gen{P_v, P_u}$ is abelian, which means that $\xtype(P_v)\cup\xtype(P_u)$ spans a clique. In other words, $\xtype(P_v)$ and $\xtype(P_u)$ are adjacent or equal in $\G\ext_k$.
\medbreak
\textbf{Step 2: Extension to hat vertices.} 
We now extend the map $\theta$ to the whole graph $\Delta$, adding all link-equivalence classes of hat vertices one at a time, by increasing size of their links. More precisely, consider a chain of induced subgraphs
\[\Delta\setminus\widehat\Delta=\Delta_0\subseteq\Delta_1\subseteq\dots\subseteq\Delta_\ell=\Delta,\] satisfying the following two conditions:
\begin{itemize}
\item for every $i\in\{1,\dots,\ell\}$, $V(\Delta_i)\setminus V(\Delta_{i-1})$ consists of exactly one link-equivalence class $[v_i]_{\lk}$ of vertices of $\Delta$;
\item if $i\le j$, then $|\lk_\Delta(v_i)|\le |\lk_\Delta(v_j)|$.
\end{itemize}
We now inductively extend $\theta$ to a map $\theta\colon V(\Delta_i)\to V(\Gamma\ext_k)$ satisfying the following property, extending Condition~\ref{condition1} which only gives the case $i=0$:

\begin{enumerate}[label = ($\bigstar$)$_i$, ref = ($\bigstar$)]
\item\label{condition} For any two vertices $u,v\in V(\Delta_i)$, the union $\theta(u)\cup\theta(v)$ spans a clique if and only if $u$ and $v$ are adjacent or equal.
\end{enumerate}

The map $\theta$ has already been defined on $\Delta_0$ in Step~1. We now let $i\in\{1,\dots,\ell\}$, and assume that $\theta$ has been defined on $\Delta_{i-1}$. We now aim to define $\theta(w)$ for every $w\in [v_i]_{\lk}$.

Let us show that there are only non-hat vertices in $\lk_\Delta(v_i)$. For $z\in\lk_\Delta (v_i)$, we have $\st_\Delta (v_i)\subseteq \st_\Delta (z)$ because $\st_\Delta(v_i)$ is a clique ($v_i$ is a hat). The graph $\Delta$ is clique reduced so $\st_\Delta (v_i)\subsetneq \st_\Delta (z)$. Thus, there exists a vertex $u\in\st_\Delta(z)\setminus\st_\Delta (v_i)$. Then, $\lk_\Delta(z)$ contains $u$ and $v_i$ that are not adjacent, so $z$ is not a hat.

As a consequence of our definition of $\theta$ on $\Delta_0$ in Step~1, $\theta(z)=\xtype(P_z)$ for any $z\in\lk_\Delta(v_i)$. All $z\in \lk_\Delta(v_i)$ are pairwise adjacent and in $\Delta_0$, thus by Condition~\ref{condition1}, the union of all $\theta(z)$, $z\in \lk_\Delta(v_i)$ induces a clique of $\Gamma\ext$. Therefore, the subgroup \[Q:=\gen{P_z}_{z\in \lk_\Delta(v_i)}\]
is a parabolic abelian subgroup of $\AG$.
\bigskip

\begin{claim}\label{Pu dans KKperp}
For every $u\in\lk_\Delta( v_i)^\perp$, one has $P_u\subseteq Q\times Q^\perp$.
\end{claim}
\begin{subproof}
Let $u\in\lk_\Delta( v_i)^\perp$. Item~\ref{witness-a} from the definition of a parabolic witness map (Definition~\ref{defn:witness}) ensures that for every $z\in\lk_\Delta( v_i)$, $P_u$ normalises $P_z$. Therefore $P_u$ normalises $Q$, i.e.\ $P_u\subseteq Q\times Q^\perp$.  
\end{subproof}

\bigskip
We will now define $\theta$ on $[v_i]_{\lk}$, distinguishing whether $Q^\perp$ is abelian or not.

\medbreak

$\drsh$ \textbf{Case A:} $Q^\perp$ is abelian. 

For every $u\in\lk_\Delta(v_i)^\perp$, Claim~\ref{Pu dans KKperp} gives $P_u\subseteq Q\times Q^\perp$. In particular, $P_{v_i}\subseteq Q\times Q^\perp$. But this means that $\gen {P_u,P_{v_i}}\subseteq Q\times Q^\perp$ which makes $\gen {P_u,P_{v_i}}$ abelian. By Item~\ref{witness-b} of the definition of a parabolic witness map, $u$ and $v_i$ are equal or adjacent. In fact they are equal because $u\notin \lk_\Delta(v_i)$. 
As this is true for every $u\in\lk (v_i)^\perp$, we deduce that \[\lk_\Delta(v_i)^\perp=\{v_i\}=[v_i]_{\lk}~,\]
    so we only need to extend $\theta$ to $v_i$.
    Let us show that letting $\theta(v_i)=\xtype(Q)$ fulfills Condition~\ref{condition}$_{i}$. Since \ref{condition}$_{i-1}$ already holds by induction, it is enough to check \ref{condition}$_{i}$ with $v=v_i$ and $u\in V(\Delta_{i-1})$.
    
    Assume first that $u$ is adjacent to $v_i$. Then, $u\in\lk_\Delta(v_i)$ so $\theta(u)=\xtype(P_u)$ is contained in $\xtype(Q)$, so $\xtype(Q)\cup \theta(u)$ spans a clique.
    
    Conversely, assuming that $\xtype(Q)\cup \theta(u)$ spans a clique, we aim to prove that $u\in\lk_\Delta(v_i)$. Our assumption ensures that $\theta(z)\cup \theta(u)$ spans a clique for any $z\in\lk_\Delta(v_i)$ because $\theta(z)\subseteq \xtype(Q)$. As both $z$ and $u$ are in $\Delta_{i-1}$ which satisfies Condition~\ref{condition}$_{i-1}$, we get that $u$ and $z$ are adjacent or equal. In particular, if $u\notin\lk_\Delta (v_i)$, then $u$ is adjacent to every $z\in\lk_\Delta( v_i)$, which means that $u\in\lk_\Delta(v_i) ^\perp=\{v_i\}$, a contradiction. Thus, $u\in\lk_\Delta(v_i)$, as desired.
    
\medbreak
$\drsh$ \textbf{Case B:} $Q^\perp$ is not abelian. 

Write $[v_i]_{\lk}=\{v_{i,1},\dots,v_{i,n}\}$. We start with the following claim.
\bigskip

\begin{claim}\label{claim theta u not in the centre}
For every $u\in V(\Delta_{i-1})\setminus \lk_\Delta(v_i)$, one has $\theta(u)\not\subseteq \xtype(Q\times Z(Q^\perp))$.
\end{claim}
\begin{subproof}
Let $u\in V(\Delta_{i-1})\setminus\lk_\Delta(v_i)$, and assume for the sake of contradiction that \[\theta(u)\subseteq \xtype(Q\times Z(Q^\perp)).\] For any $z\in\lk_\Delta(v_i)$, we have \[\theta(z)\subseteq \xtype(Q)\subseteq \xtype(Q\times Z(Q^\perp)),\]

so Condition~\ref{condition}$_{i-1}$ (given by our induction hypothesis) ensures that $u$ and $z$ are adjacent or equal. But $u\neq z$ because $u\notin\lk_\Delta(v_i)$, so $z\in\lk_\Delta(u)$. Therefore $\lk_\Delta(v_i) \subseteq \lk_\Delta(u)$.

Moreover, $u$ and $v_i$ cannot be link equivalent because $u\in V(\Delta_{i-1})$ while $v_i\in V(\Delta_i)$, so we get $\lk_\Delta(v_i)\subsetneq \lk_\Delta (u)$. Furthermore, our choice of the subgraphs $\Delta_i$ ensures that $u$ is not a hat, otherwise we would have $|\lk_\Delta(u)|\leq |\lk_\Delta(v_i)|$. Thus $\theta(u)=\xtype(P_u)$ by our definition of $\theta$ on $\Delta_0$. This gives us
        \[P_u\subseteq Q\times Z(Q^\perp).\]
        In addition, Claim~\ref{Pu dans KKperp} ensures that $P_{v_i}\subseteq Q\times Q^\perp$.
Therefore, $P_u$ is abelian and $P_u$ and $P_{v_i}$ normalise each other. This is in contradiction with Item~\ref{witness-b} of the definition of a parabolic witness map because $u$ and $v_i$ are not adjacent.
\end{subproof}
\bigskip

For every $u\in V(\Delta_{i-1})\setminus \lk_\Delta(v_i)$, Claim~\ref{claim theta u not in the centre} provides us with a vertex $x_u\in V(\G\ext)$ such that \[x_u\in V(\theta(u))\setminus V\left(\xtype\left(Q\times Z(Q^\perp)\right)\right).\]

Lemma~\ref{lemma:xtype} provides us with vertices $w_1,\dots,w_n$ of $\xtype(Q^\perp)$ such that 
\begin{itemize}
\item for $1\leq i<j\leq n$, the vertices $w_i$ and $w_j$ are distinct and non-adjacent; 
\item for $1\leq i\leq n$ and $u\in V(\Delta_{i-1})\setminus \lk_\Delta(v_i)$, the vertices $w_i$ and $x_u$ are distinct and non-adjacent. 
\end{itemize}
For every $j\in\{1,\dots,n\}$, we set $\theta(v_{i,j}):=\{w_j\}$. We now check that Condition~\ref{condition}$_{i}$ still holds for this extension of $\theta$.

We first check \ref{condition}$_{i}$ when $u,v$ both belong to $[v_i]_{\lk}$, i.e.\ $u=v_{i,j}$ and $v=v_{i,\ell}$ with $j,\ell\in\{1,\dots,n\}$. In this case $u$ and $v$ are non-adjacent, and the vertices $w_j$ and $w_\ell$ are distinct and non-adjacent, so $\theta(u)$ and $\theta(v)$ are non-adjacent. 

We are finally left with checking \ref{condition}$_{i}$ when $v=v_{i,j}$ for some $j\in\{1,\dots,n\}$ and $u\in V(\Delta_{i-1})$. For this:
\begin{itemize}
\item Assume first that $u$ is adjacent to $v_{i,j}$. Then, $\theta(u)\subseteq\xtype(Q)$ and $\theta(v_{i,j}) \subseteq\xtype(Q^\perp)$ and is reduced to one vertex. Thus, $\theta(u)\cup\theta(v_{i,j})$ spans a clique.
\item Conversely, assume that $u$ is not adjacent to $v_{i,j}$. 
Then the vertices $x_u$ and $w_j$ are distinct and non-adjacent, so $\theta(u)\cup\theta(v_{i,j})$ does not span a clique.
\end{itemize}

\medbreak
\textbf{Step 3: Injectivity.} The final step of the proof is to check that the map $\theta\colon V(\Delta)\to V(\Gamma\ext_k)$ we have constructed in Step~2 (for $i=\ell$) is injective. If $u$ and $v$ are distinct, then $\st_\Delta (u)\neq\st_\Delta (v)$ because $\Delta$ is clique reduced. Up to swapping $u$ and $v$, there exists a vertex $w$ of $\Delta$ that is
\begin{itemize}
\item adjacent to $u$,
\item distinct from $v$ (and from $u$),
\item not adjacent to $v$.
\end{itemize}
Then, by Condition~\ref{condition}$_{\ell}$, $\theta(u)$ and $\theta(w)$ are adjacent or equal, while $\theta(v)$ and $\theta(w)$ are distinct and non-adjacent. Therefore, $\theta(u)\neq \theta(v)$.
\end{proof}

\section{From graph embeddings to algebraic embeddings of RAAGs}
\label{section:graph_to_algebraic}
The goal of this section is to prove the following theorem.
\medskip

\begin{thm}
\label{thm:graph_to_algebraic}Let $\Gamma$ be a graph with no induced squares, and let $\Delta\leq \Gamma\ext_k$ be a finite graph whose vertices correspond to pairwise disjoint cliques of $\Gamma\ext$. Then $A(\Delta)\leq A(\Gamma)$.
\end{thm}

\medskip
More precisely, we prove the following rather explicit embedding theorem, which is more general in the sense that it allows $\Gamma$ to contain some induced squares. Each vertex $K$ of $\Gamma\ext_k$ corresponds to a finite clique of $\Gamma\ext$ which in turn is the extended type of a unique abelian parabolic subgroup of $A(\Gamma)$ which we denote $P_K$. In other words, $P_K$ is generated by the set of vertices of $K$ in $\Gamma\ext$, seen as elements of $A(\Gamma)$.

\begin{thm}
\label{thm:graph_to_algebraic_explicit}Let $\Gamma$ be a graph and let $\Delta\leq \Gamma\ext_k$ be a finite graph. Assume the following:
\begin{enumerate}
    \item \label{item:explicit_1}The vertices of $\Delta$ correspond to pairwise disjoint cliques of $\Gamma\ext$.
    \item \label{item:explicit_2}For every induced square $S$ of $\Gamma$, at least one of the following holds:
    \begin{itemize}
    \item some vertex of $S$ is not in any clique corresponding to a vertex of $\Delta$;
    \item no edge of $S$ belongs to a clique corresponding to a vertex of $\Delta$.
    \end{itemize}
\end{enumerate}

Then there exists an integer $M>0$ with the following property: for any family $(g_K)_{K\in V(\Delta)}$ of elements of $A(\Gamma)$ such that $\supp(g_K)=P_K$ for all $K\in V(\Delta)$, the map
%do not delete this line skip

\[\begin{aligned}
V(\Delta)&\to A(\Gamma)\\
K&\mapsto g_K^{M}
\end{aligned}\]
extends to an embedding $A(\Delta)\hookrightarrow A(\Gamma)$.
\end{thm}

\begin{rem}
\label{rem:necessary_assumptions_algebraic}
Note that Assumption~(\ref{item:explicit_2}) clearly holds as soon as $\Gamma$ has no induced squares. Therefore, Theorem~\ref{thm:graph_to_algebraic_explicit} implies Theorem~\ref{thm:graph_to_algebraic}.

Assumption~(\ref{item:explicit_1}) cannot be weakened to assume only that there is no inclusion between the cliques. Indeed, let $\Gamma = C_6$ be a cycle of length $6$ with vertices $a$, $b$, $c$, $d$, $e$, $f$ in order, and let $\Delta\leq \Gamma\ext_k$ be the subgraph induced by the six two-vertex cliques that are the edges of $\Gamma\leq \Gamma\ext$. Since $\Gamma$ has no $3$-clique, $\Delta$ is an edgeless graph on six vertices, i.e.~$A(\Delta)\simeq F_6$. However, if we choose $g_K$ to be simply the product of the elements of $K$, for all $M> 0$ we get a relation in $A(\Gamma)$ of the form $(ab)^M (bc)^{-M} (cd)^M (de)^{-M} (ef)^M (fa)^{-M} = 1$, which does not come from relations of the free group $A(\Delta)$ (other, more complicated, such relations can be found for any choice of the $g_K$).

Assumption~(\ref{item:explicit_2}) cannot be dropped either. For example, let $A(\Gamma)=F_2\times F_2$, so that $\Gamma$ is a square with vertices $a$, $b$, $c$, $d$ in order. Let $\Delta\leq \Gamma\ext_k$ be the subgraph induced by the three disjoint cliques $\{a\}$, $\{b\}$, $\{c,d\}$. Among these three cliques, only $\{a\}$ and $\{b\}$ are adjacent in $\Gamma\ext_k$, so $A(\Delta)\simeq\mathbb{Z}^2\ast\mathbb{Z}$. For all $M> 0$ we get a relation in $A(\Gamma)$ of the form $[g_{\{a\}}^M,[g_{\{c,d\}}^M,g_{\{b\}}^M]]=1$, which does not come from relations of the free product $A(\Delta)$.
\end{rem}
\medskip

We now turn to the proof of Theorem~\ref{thm:graph_to_algebraic_explicit}. To simplify the statements of intermediate lemmas, we introduce the following definitions.

\begin{defn}
Let $\Gamma$ be a graph, and let $\Delta\leq \Gamma\ext_k$. A \deffont{$\Delta$-supported family} is a family $(g_K)_{K\in V(\Delta)}$ of elements of $A(\Gamma)$ such that $\supp(g_K)=P_K$ for each $K\in V(\Delta)$.

Every $\Delta$-supported family, seen as a map $V(\Delta)\to A(\Gamma)$ has an \deffont{induced homomorphism} $A(\Delta)\to A(\Gamma)$, sending $K$ to $g_K$: this indeed defines a homomorphism because if $K,L$ are adjacent vertices of $\Delta$, then $P_K, P_L$ are contained in a common abelian parabolic subgroup, hence $g_K$ and $g_L$ commute.

Say $\Delta$ is \deffont{separated} (in $\Gamma\ext_k$) if the vertices of $\Delta$ correspond to pairwise disjoint cliques of $\Gamma\ext$. Say $\Delta$ \deffont{avoids squares} (in $\Gamma\ext_k$) if there is no induced square $S$ of $\Gamma\ext$ such that each of the vertices of $S$ and at least one edge of $S$ belong to (possibly distinct) cliques that correspond to vertices of $\Delta$.
\end{defn}

Note that $\Delta$ always avoids squares if $\Gamma$ has no induced square (in that case, $\Gamma\ext$ has no induced square by Corollary~\ref{cor:square_F2xF2}). Note also that if $\Delta$ is separated and avoids squares in $\Gamma\ext_k$, then the same goes for all induced subgraphs of $\Delta$.

Using this terminology, we can restate Theorem~\ref{thm:graph_to_algebraic_explicit} as follows:
\begin{manualtheorem}{\ref{thm:graph_to_algebraic_explicit}}[Restatement]
Let $\Gamma$ be a graph, and let $\Delta\leq \Gamma\ext_k$ be finite, separated, and avoiding squares. There exists $M>0$ such that for every $\Delta$-supported family $(g_K)_{K\in V(\Delta)}$, the homomorphism induced by the $\Delta$-supported family $(g_K^M)$ is injective.
\end{manualtheorem}

We first prove Theorem~\ref{thm:graph_to_algebraic_explicit} under the weaker assumption $\Delta\leq \Gamma_k$ (Lemma~\ref{lem:subgraph_disjoint_injective}). The proof relies on an inductive argument one of whose major steps is given by Lemma~\ref{lem:inductive_step}. Lemma~\ref{lem:intersection_2_convex} is a technical result used to check the assumptions of Lemma~\ref{lem:inductive_step}. Finally we invoke Theorem~\ref{thm:powers_Kim_Koberda} to pass from $\Delta\leq \Gamma_k$ to $\Delta\leq\Gamma\ext_k$.

\begin{defn}
Let $\Gamma$ be a graph, let $\Delta\leq \Gamma_k$, and let $\Theta\leq \Gamma$. Denote by $\Delta_{\mid \Theta}$ the subgraph of $\Delta$ induced by all cliques $K\in V(\Delta)$ of $\Gamma$ which are entirely contained in $\Theta$.

A \deffont{$\Delta$-spike} of $\Theta$ in $\Gamma$ is a triple of pairwise distinct vertices $v_0, v_1, v_2$ of $\Gamma$ such that
\begin{itemize}
    \item $v_1$ and $v_2$ are non-adjacent vertices of $\Theta$;
    \item $v_0\notin V(\Theta)$, and $v_0$ is adjacent to $v_1$ and $v_2$;
    \item the edge joining $v_0$ and $v_1$ belongs to a clique corresponding to a vertex of $\Delta$;
    \item $v_2$ belongs to a clique corresponding to a vertex of $\Delta$.
\end{itemize}
\end{defn}

The following lemma is the main technical ingredient of the proof of Theorem~\ref{thm:graph_to_algebraic_explicit}.

\begin{lem}
\label{lem:intersection_2_convex}Let $\Gamma$ be a graph and $\Delta\leq \Gamma_k\leq \Gamma\ext_k$ be separated. Let $f\colon A(\Delta)\to A(\Gamma)$ be the homomorphism induced by some $\Delta$-supported family.

Let $\Theta\leq\Gamma$. If $\Theta$ has no $\Delta$-spike in $\Gamma$, then $f(A(\Delta))\cap A(\Theta) = f(A(\Delta_{\mid \Theta}))$.
\end{lem}
The assumption that $\Theta$ has no $\Delta$-spike in $\Gamma$ cannot be dropped. Indeed, if $A(\Gamma)=\ZZ\times F_2$, the $F_2$ factor corresponds to a subgraph $\Theta\leq\Gamma$. If we choose for $\Delta$ a vertex $K_1$ and an edge $K_2$ that are disjoint, then $f(A(\Delta_{\mid \Theta}))$ is cyclic, generated by $g_{K_1}$, and the commutator $[g_{K_1}, g_{K_2}]$ is an element of $f(A(\Delta))\cap F_2$ that does not belong to $\langle g_{K_1}\rangle$. In that case, the three vertices of $\Gamma$ form a $\Delta$-spike of $\Theta$ in $\Gamma$. This counterexample is largely related to the second counterexample of Remark~\ref{rem:necessary_assumptions_algebraic}.

\begin{proof}
Let $(g_K)_{K\in V(\Delta)}$ be any $\Delta$-supported family, and let $f\colon A(\Delta)\to A(\Gamma)$ be its induced homomorphism. The inclusion $f(A(\Delta_{\mid \Theta}))\subseteq f(A(\Delta))\cap A(\Theta)$ clearly holds (without any assumption on $\Theta$). Arguing by contrapositive, assume the existence of an element $g\in (f(A(\Delta))\cap A(\Theta))\setminus f(A(\Delta_{\mid \Theta}))$. We shall construct a $\Delta$-spike of $\Theta$ in $\Gamma$. A \deffont{$\Delta$-word} for $g$ is a word in the elements of $V(\Gamma)$ and their inverses that represents $g$ and decomposes as a concatenation $w_1\cdots w_k$ where each $w_i$ is a word representing some $g_K$ or $g_K^{-1}$, $K\in V(\Delta)$. Such a word exists because $g\in G$. Let $w_\Delta = w_1\cdots w_k$ be a $\Delta$-word for $g$ with a minimal number of letters. Note that all the $w_i$ are in reduced form, otherwise reducing them would yield a shorter $\Delta$-word for $g$.
\medskip

Next, consider a reduction of $w_\Delta$ into a reduced form $w_\Theta$ (see Definition~\ref{defn:reduction}). Since $g\in A(\Theta)$, all the letters of $w_\Theta$ are elements of $V(\Theta)$ or their inverses. At each step of the reduction, we keep track of which letters came from which $w_i$ originally. Note that, by definition, all of the elementary moves of the reduction that are exchanges happen between distinct letters (even up to inversion). The rest of the proof is written in terms of this reduction, but can also be understood in terms of disc diagrams in $\cat$ cube complexes. Although we will not explain the latter point of view, the reader familiar with this type of argument can refer to Figure~\ref{fig:disk_diagram} for an illustration of the disc diagram used.
\medskip

\begin{figure}
	\centering
	\includesvg[width=0.7\textwidth]{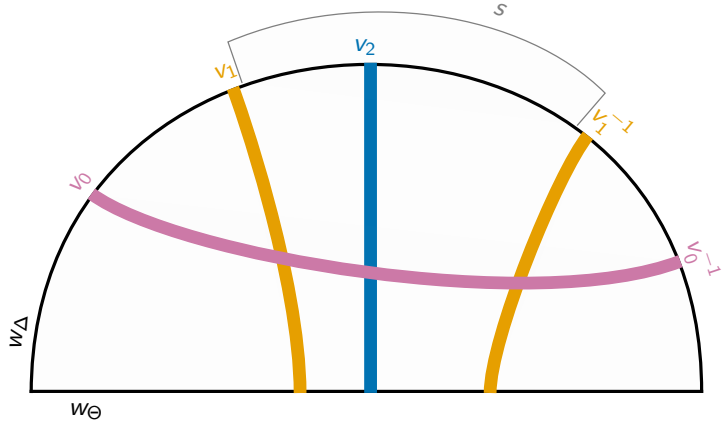}
	\caption{The disc diagram that can be used to prove Lemma~\ref{lem:intersection_2_convex}.}
	\label{fig:disk_diagram}
\end{figure}

\emph{Observation:} since the $K\in V(\Delta)$ are pairwise disjoint, for each $K$ such that one of the $w_i$ represents $g_K$ or $g_K^{-1}$, at least one vertex of $K$ appears as a letter of $w_\Theta$. Indeed, otherwise deleting all the $w_i$ representing $g_K$ or $g_K^{-1}$ from $w_\Delta$ would yield a shorter $\Delta$-word for $g$ (the word still represents $g$ because it has the same reduction to $w_\Theta$, ignoring all steps formerly involving a letter of a discarded $w_i$).
\medskip

Now, we shall find three vertices $v_0,v_1,v_2$ of $\Gamma$ satisfying the definition of $\Delta$-spike. Because $g\notin f(A(\Delta_{\mid \Theta}))$, some $w_i$ represents a $g_K$ or $g_K^{-1}$ with $K\nleq \Theta$. Since $\supp(g_K)=P_K$, which is abelian, this means that some letter $v_0\in V(\Gamma)\setminus V(\Theta)$ appears in $w_i$ with non-trivial (possibly negative) total power.
However, since $v_0$ does not appear in $w_\Theta$, this means that the total power of $v_0$ in $w_\Delta$ is $0$. Since $K$ is disjoint from all the other cliques appearing as vertices of $\Delta$, we get that the total power of each element of $K$ in $w_\Delta$ is zero. However, we know that some element $v_1$ of $K$ must still appear as a letter in $w_\Theta$ by the observation above, albeit with total power $0$ since elementary moves do not affect total powers. This means that $v_1\in V(\Theta)$. Note that the edge joining $v_0$ to $v_1$ in $\Gamma$ belongs to the clique $K$ corresponding to a vertex of $\Delta$.
\medskip

Let $E$ be the ordered set of occurrences of $v_1$ or $v_1^{-1}$ in $w_\Delta$ which are not deleted during the reduction. To each element of $E$, we associate an integer, which is the total power of $v_1$ in the prefix of $w_\Delta$ up to and including that occurrence. We obtain a sequence of integers indexed by $E$, starting at $\pm 1$, ending at $0$. Some $v_1$ or $v_1^{-1}$ might occur in $w_\Delta$ between two consecutive elements of $E$, but it can only eventually cancel with an other occurrence lying between the same two elements of $E$, because elementary exchange moves happen only between distinct letters (even up to inversion). Therefore, each subword delimited by consecutive elements of $E$ (not containing the occurrences in $E$ themselves but all the letters in between) must have total power $0$ in $v_1$. This implies that our sequence of integers must have increments of $\pm 1$. This sequence must reach an extremum at some value different from $0$, say a positive maximum (the case of a negative minimum is symmetrical). The element of $E$ where the maximum is reached, and the element of $E$ immediately following it are occurrences of $v_1$ and $v_1^{-1}$ respectively. Let $s$ be the subword of $w_\Delta$ delimited between these two occurrences. As above, $s$ has total power of $v_1$ equal to $0$. Note however that since our integer sequence takes a positive value at the occurrence of $v_1$ reaching the maximum, the largest prefix of $w_\Delta$ before $s$ has positive total power in $v_1$.
\medskip

Finally, $s$ must contain some occurrence of a letter $v_2$ which does not commute with $v_1$ and is not deleted during the reduction. Indeed, otherwise the occurrences of $v_1$ and $v_1^{-1}$ surrounding $s$ would yield occurrences of $v_1$ and $v_1^{-1}$ in $w_\Theta$ only separated by letters commuting with $v_1$, i.e.~cancellable with each other, contradicting the fact that $w_\Theta$ is reduced. In particular, $v_2$ appears in $w_\Theta$, i.e.~$v_2\in V(\Theta)$. This occurrence of $v_2$ does not appear between occurrences of $v_1$ and $v_1^{-1}$ in $s$ which eventually cancel together, because this occurrence of $v_2$ can never be deleted or commuted away from between them. This means that the occurrences of $v_1$ and $v_1^{-1}$ in $s$ before (resp. after) this occurrence of $v_2$ eventually cancel in pairs together, thus the prefix (resp. suffix) of $s$ before (resp. after) this occurrence of $v_2$ has total sum in $v_1$ equal to $0$. Now, this occurrence of $v_2$ belongs to some $w_\ell$ representing a $g_L$ or $g_L^{-1}$ with $L\neq K$, because $K$ and $L$ are cliques and $v_1$ and $v_2$ do not commute. We know that the total power of $v_1$ in $w_1\cdots w_{\ell-1}$ is positive (it is positive before $s$ and zero in $s$ before $w_\ell$). Say without loss of generality that $v_0$ and $v_1$ both appear with positive powers in $g_K$ (the other cases are symmetrical). This means that more of the $w_j$, $1\leq j\leq \ell-1$ represent $g_K$ than $g_K^{-1}$, thus $v_0$ has positive total power in $w_1\cdots w_{\ell-1}$. Since $v_0$ does not appear in $w_\Theta$ nor in $w_\ell$, this means that some occurrence of $v_0$ in $w_1\cdots w_{\ell-1}$ must eventually cancel with some occurrence of $v_0^{-1}$ in $w_{\ell+1}\cdots w_k$. To do so, one of them must commute with the occurrence of $v_2$ in $w_\ell$ during one of the steps of the reduction, or else this occurrence of $v_2$ would separate them throughout the reduction. Therefore $v_0$ and $v_2$ commute. Note also that $v_2$ belongs to the clique $L$ of $\Gamma$, corresponding to a vertex of $\Delta$. Moreover, we have proved that $v_1$ and $v_2$ are distinct, non-adjacent vertices of $\Theta$, but both adjacent to $v_0$, a vertex of $\Gamma\setminus \Theta$, thus they form a $\Delta$-spike.
\end{proof}

The following lemma relies on normal forms for HNN extensions.

\begin{lem}
\label{lem:inductive_step}Let $\Gamma$ be a graph, and $\Delta\leq \Gamma_k\leq \Gamma\ext_k$ be separated. Assume the following:
\begin{enumerate}
    \item \label{item:inductive_step_1}Some vertex $K_0$ of $\Delta$ is a one-element clique $\{v\}$.
    \item Letting $\Delta' \coloneqq \Delta\setminus \{K_0\}$, the following hold:
    \begin{enumerate}
    \item\label{item:inductive_step_2}  any homomorphism $f'\colon A(\Delta')\to A(\Gamma)$ induced by some $\Delta'$-supported family satisfies
    \[f'(A(\Delta'))\cap A(\lk_\Gamma(v)) = f'\left(A\left(\Delta'_{\mid \lk_\Gamma(v)}\right)\right);\]
    \item \label{item:inductive_step_3} any homomorphism $A(\Delta')\to A(\Gamma)$ induced by some $\Delta'$-supported family is injective.
\end{enumerate}
\end{enumerate}
Then, any homomorphism $A(\Delta)\to A(\Gamma)$ induced by a $\Delta$-supported family is injective.
\end{lem}

\begin{rem}
\label{rem:proper_implies_no_spike} By Lemma~\ref{lem:intersection_2_convex} applied with $\Gamma\setminus \{v\}$ and $\Delta'$, Assumption~(\ref{item:inductive_step_2}) holds as soon as $\lk_\Gamma(v)$ has no $\Delta'$-spike in $\Gamma\setminus\{v\}$. This is immediately satisfied if $\Delta$ avoids squares (in particular if $\Gamma$ has no induced square). Indeed, arguing by contrapositive, assume we have a $\Delta'$-spike of $\lk_\Gamma(v)$ in $\Gamma\setminus \{v\}$, given by three vertices $v_0, v_1, v_2$. This defines an induced square of $\Gamma$ with vertices $v_1, v_0, v_2, v$ in that order. By definition of a $\Delta'$-spike, and by (\ref{item:inductive_step_1}), the edge joining $v_1$ to $v_0$, the vertex $v_2$ and the vertex $v$ all belong to (distinct) cliques that correspond to vertices of $\Delta$, showing that $\Delta$ does not avoid squares.
\end{rem}
\begin{proof}
Let $(g_K)_{K\in V(\Delta)}$ be a $\Delta$-supported family, with induced homomorphism $f$. Write $g_{K_0} = v^m$. The restriction $f'$ of $f$ to $A(\Delta')\leq A(\Delta)$ is induced by the $\Delta'$-supported family $(g_K)_{K\neq K_0}$. Thus, by Assumption~(\ref{item:inductive_step_3}), $f'$ is injective.

Let $g\in A(\Delta)$ be non-trivial. If $g\in A(\Delta')$, then $f(g)=f'(g)$ is non-trivial. Otherwise, we use the ``reduced'' normal forms for HNN extensions from \cite[Chapter~IV, Section~2, Britton's Lemma]{Lyndon_Schupp}. We have the following splittings of $A(\Gamma)$ and $A(\Delta)$ as HNN extensions over the links of $v$ and $K_0$ respectively, where in each case the two embeddings of the edge group in the vertex group are the same standard parabolic embedding:
\[\begin{aligned}A(\Gamma) &\simeq A(\Gamma\setminus \{v\}) \ast_{A(\lk_\Gamma(v))}\text{, with stable letter }v,\\
A(\Delta) &\simeq A\left(\Delta'\right) \ast_{A(\lk_\Delta(K_0))} = A\left(\Delta_{\mid \Gamma\setminus \{v\}}\right) \ast_{A\left(\Delta_{\mid K_0^\perp}\right)}\text{, with stable letter }K_0.
\end{aligned}\]
We write
\[g = g_0 K_0^{\alpha_1} g_1\dots K_0^{\alpha_k} g_k,\]
with $k\geq 1$, the $\alpha_i$ in $\ZZ\setminus \{0\}$ and the $g_i$ in $A(\Delta')$, such that none of the $g_i$ are in $A\left(\Delta_{\mid K_0^\perp}\right)$, except possibly $g_0$ or $g_k$. Thus, since $f(K_0) = v^m$, we have:
\[f(g) = f'(g_0) v^{m\alpha_1} f'(g_1) \dots v^{m\alpha_k}f'(g_k). \tag{$\bigstar$}\label{HNN_normal_form}\]

Note that $k\geq 1$ because $g\notin A(\Delta')$, the $m\alpha_i$ are non-zero, and the $f'(g_i)$ are in $A(\Gamma\setminus\{v\})$. To obtain injectivity of $f$, we wish to prove that (\ref{HNN_normal_form}) is a normal form for $f(g)$. 

To do so, assume for the sake of contradiction that there exists $i_0\in\{1,\dots,k-1\}$ such that $f'(g_{i_0})$ belongs to the edge group $A(\lk_\Gamma(v))$. By Assumption~(\ref{item:inductive_step_2}), we have:
\[f'(A(\Delta'))\cap A(\lk_\Gamma(v)) = f'\left(A\left(\Delta'_{\mid \lk_\Gamma(v)}\right)\right)= f'\left(A\left(\Delta_{\mid \lk_\Gamma(v)}\right)\right).\]

Since $g_{i_0}\in A(\Delta')$, we have  $f'(g_{i_0})\in f'(A(\Delta'))$. Thus, the equality above yields $f'(g_{i_0}) \in  f'\left(A\left(\Delta_{\mid \lk_\Gamma(v)}\right)\right)$. By injectivity of $f'$, we have $g_{i_0}\in A\left(\Delta_{\mid \lk_\Gamma(v)}\right)$. Yet, since $K_0 = \{v\}$, $K_0^\perp = \lk_\Gamma(v)$. Thus  $g_{i_0}\in A\left(\Delta_{\mid K_0^\perp}\right)$, contradicting that we started with a normal form for $g$.

Therefore, none of the $f'(g_i)$, $0<i<k$ belong to $A(\lk_\Gamma(v))$, meaning that (\ref{HNN_normal_form}) is a non-trivial normal form for $f(g)$ in the HNN splitting of $A(\Gamma)$. Hence, $f(g)$ is non-trivial.
\end{proof}

Now we can establish a version of Theorem~\ref{thm:graph_to_algebraic_explicit} which does not require to pass to a power, in the special case where $\Delta\leq \Gamma_k$.

\begin{lem}
\label{lem:subgraph_disjoint_injective}Let $\Gamma$ be a graph and let $\Delta\leq \Gamma_k\leq \Gamma\ext_k$ be finite, separated, and avoiding squares. Then, any $\Delta$-supported family induces an injective homomorphism $A(\Delta)\hookrightarrow A(\Gamma)$.
\end{lem}
\begin{proof}
We fix $\Gamma$ and proceed by induction on the number of vertices of $\Delta$. The result clearly holds when $\Delta$ has at most one vertex, since $A(\Gamma)$ is torsion free.

Let $n>1$, assume the result holds for all graphs with at most $n-1$ vertices, and let $\Delta$ have $n$ vertices. Let $K_0$ be a vertex of $\Delta$. Viewing $K_0$ as a clique of $\Gamma$, we let $k_1,\dots, k_m\in V(\Gamma)$ be the vertices of $K_0$ ($m\geq 1$). Let $\widetilde \Delta$ be the subgraph of $\Gamma_k$ induced by $(V(\Delta)\setminus \{K_0\}) \cup \{\{k_1\},\dots, \{k_m\}\}$. We shall factor each induced homomorphism $A(\Delta)\to A(\Gamma)$ into a composition of induced homomorphisms $A(\Delta)\to A(\widetilde \Delta)\to A(\Gamma)$.

\textbf{Step 1: $\widetilde\Delta$ is separated and avoids squares in $\Gamma_k$.} The graph $\widetilde \Delta$ is separated in $\Gamma_k$ by construction. Further, if all the vertices, and an edge of some induced square of $\Gamma$ belonged to (possibly distinct) cliques corresponding to vertices of $\widetilde \Delta$, then the same would clearly be true of $\Delta$ with the same square. Therefore, $\widetilde \Delta$ avoids squares in $\Gamma_k$.

\textbf{Step 2: $\Delta$ is separated and avoids squares in $\widetilde \Delta_k$.} Define a map $\varphi \colon V(\Delta)\to V(\widetilde \Delta_k)$ as follows: for a vertex $K\neq K_0$ of $\Delta$, set $\varphi(K) \coloneqq \{K\}$, and set $\varphi(K_0)$ to be the $m$-element clique $\mathcal{K}$ of $\widetilde \Delta$ with vertex set $\{\{k_1\},\dots, \{k_m\}\}$. Clearly, these cliques are pairwise disjoint. Moreover, let $K, L\in V(\Delta)$ be different from $K_0$. Then, $\{K\}$ and $\{L\}$ are adjacent in $\widetilde \Delta_k$ if and only if $K$ and $L$ are adjacent in $\widetilde \Delta$, if and only if $K$ and $L$ are adjacent in $\Delta$. Likewise, $\{K\}$ and $\mathcal{K}$ are adjacent in $\widetilde \Delta_k$ if and only if $K$ is adjacent to each $\{k_i\}$ in $\widetilde \Delta$, if and only if $K$ is adjacent to $K_0$ in $\Delta$. Therefore, $\varphi$ extends to an embedding of $\Delta$ in $\widetilde \Delta_k$ as a separated induced subgraph.

To prove that $\varphi(\Delta)$ avoids squares in $\widetilde \Delta_k$, assume for the sake of contradiction that $\widetilde \Delta$ has vertices $v_0$, $v_1$, $v_2$, $v_3$ spanning an induced square in that order, such that the edge $e$ joining $v_0$ to $v_1$, the vertex $v_2$, and the vertex $v_3$ all belong to (possibly distinct) cliques corresponding to vertices of $\varphi(\Delta)$. Since $\widetilde \Delta\leq \Gamma_k$, by Corollary~\ref{cor:square_F2xF2}, we can choose a vertex $x_i$ of $\Gamma$ in each clique $v_i$ so that they span an induced square of $\Gamma$ in the same order. Now by assumption, each $v_i$, as well as $e$, is contained in some $\varphi(K)$ with $K\in V(\Delta)$. This means that each $v_i$ is either some $K\in V(\Delta)\setminus \{K_0\}$ or some vertex $\{k_j\}$ of $\mathcal{K}$, and the edge $e$ is an edge of $\mathcal{K}$ (since $\mathcal{K}$ is the only clique in $\varphi(V(\Delta))$ with possibly more than one element). Thus, $x_0$ and $x_1$ are of the form $k_j$ and $k_{j'}$ for $1\leq j<j'\leq m$: the edge joining $x_0$ to $x_1$ is contained in the clique $K_0$ of $\Gamma$. Likewise, $x_2$ and $x_3$ are also contained in cliques of $\Gamma$ which come from vertices of $\Delta$. Hence, the square induced by $x_0$, $x_1$, $x_2$, $x_3$ contradicts the assumption that $\Delta$ avoids squares in $\Gamma_k$.

\textbf{Step 3: Factorisation.} Now let $(g_K)_{K \in V(\Delta)}$ be a $\Delta$-supported family of elements of $A(\Gamma)$ with induced homomorphism $f\colon A(\Delta)\to A(\Gamma)$. Write $g_{K_0} = k_1^{\alpha_1}\cdots k_m^{\alpha_m}$, where the $\alpha_i$ are non-zero. Then, setting $g^{(1)}_K=K$ for $K\in V(\Delta)\setminus \{K_0\}$ and $g^{(1)}_{K_0} = \{k_1\}^{\alpha_1}\cdots \{k_m\}^{\alpha_m}$ defines a $\Delta$-supported family of elements of $A(\widetilde \Delta)$ (compatible with the embedding $\varphi$ defined above). Let $f^{(1)}\colon A(\Delta)\to A(\widetilde \Delta)$ denote the induced homomorphism.

Besides, setting $g^{(2)}_K = g_K$ for $K\in V(\Delta)\setminus \{K_0\}$ and $g^{(2)}_{\{k_i\}} = k_i$ for $1\leq i\leq m$ defines a $\widetilde \Delta$-supported family of elements of $A(\Gamma)$. Let $f^{(2)}\colon A(\widetilde \Delta)\to A(\Gamma)$ denote the induced homomorphism. It is clear on generators that $f = f^{(2)}\circ f^{(1)}$. It remains to see that $f^{(1)}$ and $f^{(2)}$ are both injective.

\textbf{Step 4: Injectivity of factors.} Let $K_0, v_1,\dots, v_{n-1}$ denote the vertices of $\Delta$. All the $v_i$ are mapped by $\varphi$ to one-element cliques of $\widetilde \Delta_k$. For $0\leq i\leq n-1$, let $\Delta_i$ denote the subgraph of $\Delta$ induced by $\{K_0\}\cup \{v_1,\dots, v_i\}$. Each $\Delta_i$, as an induced subgraph of $\Delta$, is separated and avoids squares in $\widetilde\Delta_k$. We prove iteratively that the restriction of $f^{(1)}$ to $\Delta_i$ is injective. This is true for $\Delta_0$ which is the single vertex $K_0$, since $A(\widetilde\Delta)$ is torsion free. If it is true for $i-1$, it is true for $i$ by Lemma~\ref{lem:inductive_step} applied to $\Delta_i$ and the one-element clique given by $v_i$: Assumption~(\ref{item:inductive_step_2}) is satisfied by Remark~\ref{rem:proper_implies_no_spike} since $\Delta_i$ avoids squares. As $\Delta_{n-1} = \Delta$, the homomorphism $f^{(1)}$ is injective.

The argument for $f^{(2)}$ is similar: for $0\leq i\leq m$, let $\widetilde\Delta_i$ be the subgraph of $\widetilde\Delta$ induced by $(\Delta\setminus\{K_0\}) \cup \{\{k_1\},\dots,\{k_i\}\}$, which is separated and avoids squares in $\Gamma_k$, being an induced subgraph of $\widetilde\Delta$. By the inductive hypothesis applied to the graph with $n-1$ vertices $\widetilde\Delta_0 = \Delta\setminus\{K_0\}$, the restriction of $f^{(2)}$ to $\widetilde\Delta_0$ is injective. Once again if the restriction of $f^{(2)}$ to $\widetilde\Delta_{i-1}$ is injective, the same holds for $\widetilde\Delta_i$ by Lemma~\ref{lem:inductive_step} applied to $\widetilde\Delta_i$ and the one-element clique $\{k_i\}$: Assumption~(\ref{item:inductive_step_2}) is satisfied by Remark~\ref{rem:proper_implies_no_spike} since $\widetilde\Delta_i$ avoids squares. As $\widetilde\Delta_m = \widetilde\Delta$, the homomorphism $f^{(2)}$ is injective.
\end{proof}

We can now finish the proof of the theorem.

\begin{proof}[Proof of Theorem~\ref{thm:graph_to_algebraic_explicit}]
Let $\Gamma$ be a graph and $\Delta\leq \Gamma\ext_k$ be finite, separated, and avoiding squares. Let $(g_K)_{K\in V(\Delta)}$ be a $\Delta$-supported family. 

Let $\Gamma'$ be the finite subgraph of $\Gamma\ext$ induced by the union of all the cliques appearing as vertices of $\Delta$. By Theorem~\ref{thm:powers_Kim_Koberda} applied to $\Gamma'$ and $\Gamma$, there exists $M>0$ such that the map
\[\begin{aligned}
V(\Gamma')&\to A(\Gamma)\\
g&\mapsto g^M
\end{aligned}\]
extends to an embedding $f\colon A(\Gamma')\hookrightarrow A(\Gamma)$. Clearly, $\Delta\leq \Gamma'_k$, and $\Delta$ is separated and avoids squares in $\Gamma'_k$. For $K\in V(\Delta)$, let $P'_K$ denote the abelian parabolic subgroup of $A(\Gamma')$ generated by the vertices of $K$, seen as a clique of $\Gamma'$. The element $g_K\in A(\Gamma)$ decomposes as a product $v_1^{\alpha_1}\cdots v_n^{\alpha_n}$ where the $v_i$ are vertices of $K$, seen in $\Gamma\ext$, and the $\alpha_i$ are non-zero. Note that all vertices of $K$ appear since $\supp(g_K) = P_K$. To avoid confusion, let $h_K\in A(\Gamma')$ denote also the product $v_1^{\alpha_1}\cdots v_n^{\alpha_n}$, but seen as a product of powers of standard generators of $A(\Gamma')$. Note that since all vertices of $K$ appear, $\supp(h_K) = P'_K$.
Thus, $(h_K)_{K\in V(\Delta)}$ is a $\Delta$-supported family of elements of $A(\Gamma')$.

By Lemma~\ref{lem:subgraph_disjoint_injective}, applied to $\Delta$, $\Gamma'$, and $(h_K)$, the map \[\begin{aligned}
V(\Delta)&\to A(\Gamma')\\
K&\mapsto h_K
\end{aligned}\]
extends to an embedding $A(\Delta)\hookrightarrow A(\Gamma')$. Composing this map with $f$, we get an embedding $A(\Delta)\hookrightarrow A(\Gamma)$ mapping $K$ to $f(h_K)$. Finally, by definition of $f$, we have $f(h_K) = v_1^{M\alpha_1}\cdots v_n^{M\alpha_n} = (v_1^{\alpha_1}\cdots v_n^{\alpha_n})^M = g_K^M$, since $K$ is a clique.
\end{proof}

\section{Conclusion}\label{sec:conclusion}

In this final section, we complete the proof of our main theorem, which we now restate.

\begin{manualtheorem}{\ref{thm:main}}
Let $\Delta,\Gamma$ be finite graphs, and assume that $\Gamma$ has no induced squares. Then the following are equivalent:
\begin{enumerate}
    \item\label{dup_main-1} $A(\Delta)$ measurably embeds into $A(\Gamma)$;
    \item\label{dup_main-1,5} there exist free measure-preserving actions of $A(\Delta)$ and $A(\Gamma)$ on a standard probability space $X$ such that $A(\Delta)\cdot x\subseteq A(\Gamma)\cdot x$ for almost every $x\in X$; 
    \item\label{dup_main-3} $A(\Delta)$ embeds as a subgroup in a graph product of free abelian groups over $\Gamma$;
    %\item\label{dup_main-2} $A(\Delta)$ embeds as a subgroup in the graph product of $\mathbb{Z}^{|V(\Delta)|}$ over $\Gamma$;
    \item\label{dup_main-4} $\clr(\Delta)$ embeds as an induced subgraph in $\Gamma\ext_k$.
\end{enumerate}
Additionally, there exists an algorithm which, given two finite graphs $\Delta,\Gamma$ as above, decides whether or not any of the above equivalent conditions holds. 
\end{manualtheorem}

\begin{proof}
The fact that Condition~(\ref{dup_main-4}) is algorithmic follows from Corollary~\ref{cor:algorithmic_cliques} and from the fact that the clique reduction of a finite graph can be computed algorithmically. For the equivalence, we prove (\ref{dup_main-4})$%\Rightarrow$(\ref{dup_main-2})$
\Rightarrow$(\ref{dup_main-3})$\Rightarrow$(\ref{dup_main-1,5})$\Rightarrow$(\ref{dup_main-1})$\Rightarrow$(\ref{dup_main-4}).

We start with (\ref{dup_main-4})$\Rightarrow$(\ref{dup_main-3}). Under the assumption that $\clr(\Delta)\leq\Gamma\ext_k$, by Lemma~\ref{lem:disjoint_cliques}, there exists a clique expansion $\widetilde\Gamma$ of $\Gamma$ (of order $|V(\Delta)|$) such that $\Delta\leq \widetilde\Gamma\ext_k$ and such that the vertices of $\Delta$ correspond to pairwise disjoint cliques of $\widetilde\Gamma\ext$. Since a square is clique reduced, and $\Gamma$ has no induced squares, by Lemma~\ref{lem:square_expansion}, $\widetilde\Gamma$ has no induced squares. By Theorem~\ref{thm:graph_to_algebraic}, we have $A(\Delta)\leq A(\widetilde\Gamma)$ and (\ref{dup_main-3}) holds.

%The implication (\ref{dup_main-2})$\Rightarrow$(\ref{dup_main-3}) is clear. 

For (\ref{dup_main-3})$\Rightarrow$(\ref{dup_main-1,5}), notice that if $A(\Delta)$ embeds in such a graph product, we can always assume that its vertex groups are finitely generated and non-trivial, i.e.\ of the form $\mathbb{Z}^n$ with $n\geq 1$. The implication (\ref{dup_main-3})$\Rightarrow$(\ref{dup_main-1,5}) thus follows from the fact that every graph product of non-trivial finitely generated free abelian groups is orbit equivalent to $A(\Gamma)$ (Theorem~\ref{thm:amenable_graph_product}).

The implication (\ref{dup_main-1,5})$\Rightarrow$(\ref{dup_main-1}) is a general fact about measurable embeddings, see Remark~\ref{rem:oe-me}.

Finally, we prove (\ref{dup_main-1})$\Rightarrow$(\ref{dup_main-4}). Since $A(\clr(\Delta))$ is orbit equivalent to $A(\Delta)$ (Corollary~\ref{cor:ME_clique_reduction}), $A(\clr(\Delta))$ measurably embeds in $A(\Delta)$ (see Remark~\ref{rem:oe-me}). Thus (\ref{dup_main-1}) combined with the transitivity of measurable embeddings (Remark~\ref{transitivité du plongement ME}) yields that $A(\clr(\Delta))$ measurably embeds into $A(\Gamma)$. Theorem~\ref{th meas embedding to graph} concludes that $\clr(\Delta)\leq\Gamma\ext_k$, i.e.\ (\ref{dup_main-4}) holds.
\end{proof}
\medskip

\begin{rem}
\label{rem:finite_unnecessary}
To prove the equivalence, the assumption that $\Gamma$ is finite is only used for (\ref{dup_main-1})$\Rightarrow$(\ref{dup_main-4}). The authors do not know whether it can be dropped. The implications (\ref{dup_main-3})$\Rightarrow$(\ref{dup_main-1,5})$\Rightarrow$(\ref{dup_main-1}) hold when $\Gamma$ is countable, and the implications (\ref{dup_main-4})$%\Rightarrow$(\ref{dup_main-2})$
\Rightarrow$(\ref{dup_main-3}) hold without any cardinality restriction on $\Gamma$.

Further, (\ref{dup_main-3})$\Rightarrow$(\ref{dup_main-4}) also holds without any cardinality restriction on $\Gamma$. Indeed, if $A(\Delta)$ embeds as a subgroup in a graph product of free abelian groups of $\Gamma$, then we can assume that these free abelian groups are countable and non-trivial, as in the proof of Theorem~\ref{thm:main}. In that case, there exists a clique expansion $\widetilde\Gamma$ of $\Gamma$ of countable order such that $A(\Delta)$ embeds as a subgroup of $A(\widetilde\Gamma)$. By Theorem~\ref{thm:main_Kim_Koberda}, we have $\clr(\Delta)\leq \Delta\leq \widetilde\Gamma\ext_k$. By Lemma~\ref{lem:clique_graph_expansion} and Remark~\ref{rem:clique-map}, the graph $\widetilde\Gamma\ext_k$ is a clique expansion of $\Gamma\ext_k$. Therefore, since $\clr(\Delta)$ is clique reduced, $\clr(\Delta)\leq \Gamma\ext_k$ by Lemma~\ref{lem:square_expansion}.
\end{rem}

\medskip
\begin{rem}\label{rem:square}
In Theorem~\ref{thm:main}, we do not make the assumption that $\Delta$ has no induced squares. However, it follows from the theorem that if $A(\Delta)$ measurably embeds into $A(\Gamma)$ and $\Gamma$ has no induced squares, then the same holds for $\Delta$. Indeed, in that case, by Corollary~\ref{cor:square_F2xF2}, $\Gamma\ext_k$ has no induced squares, hence $\clr(\Delta)$ has no induced squares. Since a square is clique reduced, by Lemma~\ref{lem:square_expansion}, $\Delta$ has no induced squares.

As a consequence, if $\Gamma$ is a finite graph and $F_2\times F_2$ measurably embeds into $A(\Gamma)$, then $F_2\times F_2$ embeds as a subgroup of $A(\Gamma)$.
\end{rem}

In the same spirit as the remark above, we deduce the following corollary. A group is \deffont{coherent} if all of its finitely generated subgroups are finitely presented.
\begin{cor}
\label{cor:coherent}Let $\Gamma, \Delta$ be finite graphs. Assume that $A(\Delta)$ mesurably embeds into $A(\Gamma)$, and $A(\Gamma)$ is coherent. Then $A(\Delta)$ is coherent.
\end{cor}
\begin{proof}
A graph is \deffont{chordal} if it has no induced cycle of length at least $4$.
By a result of Droms \cite{Droms_coherence}, a graph is chordal if and only if the associated RAAG is coherent. Let $\Delta, \Gamma$ be as in the statement: $\Gamma$ is chordal. In particular, $\Gamma$ has no induced squares. By Theorem~\ref{thm:main}, (\ref{dup_main-1})$\Rightarrow$(\ref{dup_main-3}), $A(\Delta)$ embeds as a subgroup of $A(\widetilde \Gamma)$, for some clique expansion $\widetilde \Gamma$ of $\Gamma$. Since cycles of length at least $4$ are clique reduced, $\widetilde \Gamma$ is chordal as well by Lemma~\ref{lem:square_expansion}. Thus, $A(\widetilde \Gamma)$ is coherent, and so is its subgroup $A(\Delta)$.
\end{proof}

\begin{cor}
\label{cor:no_universal}There is no finite graph $\Gamma$ without induced squares such that for every finite graph $\Delta$ without induced squares, $A(\Delta)$ measurably embeds into $A(\Gamma)$.
\end{cor}

\begin{proof}
Let $\Gamma$ be a finite graph with no induced square, and assume for the sake of contradiction that for every finite graph $\Delta$ with no induced square, $A(\Delta)$ measurably embeds into $A(\Gamma)$. In particular this holds whenever $\Delta$ is connected and has girth at least $5$, in which case $\Delta$ is either an edge or clique reduced. Theorem~\ref{thm:main} implies that for any such $\Delta$, one has $\Delta\leq\Gamma\ext_k$. This contradicts Lemma~\ref{lem:universal}.  
\end{proof}
\medskip

\begin{ex}\label{ex:square}
The following example shows that the implication (\ref{dup_main-4})$\Rightarrow$(\ref{dup_main-3}) from Theorem~\ref{thm:main} does not hold in general if $\Gamma$ is allowed to contain a square. 

Let $\Gamma$ be a square with four vertices $w,x,y,z$ (in this order), and let $\Delta$ be the line graph on four vertices $a,b,c,d$ (in this order).

Then $\Delta\leq\Gamma_k\leq \Gamma\ext_k$, as shown by the map 
\[a\mapsto w, \quad
b\mapsto  x, \quad
c  \mapsto  y, \quad
d  \mapsto \{y,z\}.\]

On the other hand, for any $n\in\NN$, the group $A(\Delta)$ does not embed as a subgroup in $G_n:=(\ZZ^n\ast\ZZ^n)\times (\ZZ^n\ast\ZZ^n)$, which is the graph product of $\mathbb{Z}^n$ over $\Gamma$. Indeed, assume that such an embedding $f$ exists. Write $G_n=F\times F'$ with both $F,F'$ isomorphic to $\ZZ^n\ast\ZZ^n$. 
\medskip

The centraliser of any non-trivial element $g\in \ZZ^n\ast\ZZ^n$ is abelian, and either conjugate to one of the two $\ZZ^n$ free factors (if $g$ is conjugate into such a free factor), or else isomorphic to $\ZZ$, in fact equal to the unique maximal cyclic subgroup of $\ZZ^n\ast\ZZ^n$ containing $g$. In particular, if $g,h,k\in\ZZ^n\ast\ZZ^n$ are non-trivial elements such that $[g,h]=1$ and $[g,k]=1$, then $[h,k]=1$ (one says that $\ZZ^n\ast\ZZ^n$ is commutative transitive). In other words, if $g,h\in F$ are non-trivial and commute, then $Z_F(g) = Z_F(h)$. Thus $Z_{G_n}((g,e)) = Z_F(g)\times F' = Z_F(h)\times F' = Z_{G_n}((h,e))$. This proves that any two non-trivial commuting elements of $F\times \{e\}$ have the same centraliser in $G_n$ (and the same holds symmetrically for $\{e\}\times F'$).

\medskip
Write $f(b)=(g,h)\in F\times F'$. Then, $Z_{G_n}(f(b)) = Z_F(g)\times Z_{F'}(h)$. Since $f$ is injective and the centraliser of $b$ in $A(\Delta)$ is non-abelian, one of $g,h$ must be trivial by the observation above. This means that $f(b)$ is either in $F\times \{e\}$ or in $\{e\}\times F'$. The same is true of $f(c)$. Now, $f(b)$ and $f(c)$ cannot be both in, say, $F\times \{e\}$ because they commute and have distinct centralisers in $G_n$: $f(a)$ is in one but not the other. Without loss of generality, say $f(b)\in F\times \{e\}$ and $f(c)\in \{e\}\times F'$.

In particular, $[f(a),f(c)]\in \{e\}\times F'$ and $[f(b),f(d)]\in F\times \{e\}$. It follows that
\[\left[[f(a),f(c)],[f(b),f(d)]\right]=1\] while $[[a,c],[b,d]]\neq 1$ (the obvious word representing the left side is reduced), in contradiction with the injectivity of $f$.
\end{ex}
\medskip

\begin{q}
Let $\Delta$ be the line graph on four vertices, and $\Gamma$ be the square graph. Does $A(\Delta)$ measurably embed in $A(\Gamma)$?

Note that by a theorem of Rull \cite{Rul}, $A(\Delta)$ quasi-isometrically embeds into $A(\Gamma)$.
\end{q}

\bigskip
We finish this section with a statement that records situations where a measurable embedding of $A(\Delta)$ into $A(\Gamma)$ forces $A(\Delta)$ to embed as a subgroup in $A(\Gamma)$; this is in contrast with the example given in Remark~\ref{rem:meas-not-alg}.

\begin{thm}
\label{thm:meas-algebraic}
Let $\Delta,\Gamma$ be finite graphs. Assume that 
\begin{itemize}
\item every vertex preimage of the clique map $\Delta\to \clr(\Delta)$ is either a vertex or an isolated edge, and for every vertex $v\in V(\Delta)$, the link $\lk_\Delta(v)$ is not a clique with at least two vertices;
\item $A(\Gamma)$ is two-dimensional and not isomorphic to $\mathbb{Z}\times F_n$ for any $n\geq 1$, and $\Gamma$ has no induced square. 
\end{itemize}
Then $A(\Delta)$ measurably embeds into $A(\Gamma)$ if and only if $A(\Delta)$ embeds as a subgroup in $A(\Gamma)$.
\end{thm}

Note that the assumptions on $\Delta$ are satisfied whenever $A(\Delta)$ is two-dimensional, thereby recovering Theorem~\ref{thm:intro-meas-algebraic} from the introduction. They also cover other situations, for instance the case where $\Delta$ is clique reduced and has no hat vertices.

\begin{proof}
Assume that $A(\Delta)$ measurably embeds into $A(\Gamma)$. Let $\Delta'\leq\Delta$ be the subgraph obtained by removing all connected components which are edges (say there are $n$ of them). Then $A(\Delta')$ still measurably embeds into $A(\Gamma)$, and our assumption on $\Delta$ ensures that $\Delta'$ is clique reduced. Propositions~\ref{prop:parabolic_support} and~\ref{prop:witness-to-graph-warmup} ensure that $\Delta'\leq\Gamma\ext$. 

We claim that we can find $n$ edges $e_1,\dots,e_n$ of $\Gamma\ext$ such that 
\begin{itemize}
\item for $i\neq j$, no vertex of $e_i$ is adjacent to a vertex of $e_j$;
\item for every $i\in\{1,\dots,n\}$, no vertex of $e_i$ is adjacent to a vertex of $\Delta'$.
\end{itemize}
Indeed, the assumptions on $\Gamma$ ensure that $\Gamma$ does not split non-trivially as a join. If $\Gamma$ is disconnected, our claim follows from the fact that $\Gamma\ext$ has infinitely many connected components that contain edges (Lemma~\ref{lem:extension_properties}). If $\Gamma$ is connected, our claim follows from the fact that $\Gamma\ext$ is connected and of infinite diameter (Lemma~\ref{lem:extension_properties}).

Now the above claim ensures that the embedding $\Delta'\leq\Gamma\ext$ extends to an embedding $\Delta\leq\Gamma\ext$. It thus follows from Theorem~\ref{thm:main_Kim_Koberda} that $A(\Delta)\leq A(\Gamma)$.
\end{proof}

\bigskip

\small

\bibliography{bib}
\bibliographystyle{alpha}

\vfill

\normalsize

	\begin{flushleft}
		\textcolor{DeepSkyBlue4}{Adrien Abgrall}\\
		ENS de Lyon,  Unité de Mathématiques Pures et Appliquées, 69007, Lyon, France \\
		\emph{e-mail:~}\texttt{firstname.lastname@ens-lyon.fr}\\[4mm]
	\end{flushleft}

	\begin{flushleft}
		\textcolor{DeepSkyBlue4}{Alexandra Gurieva}\\
		ENS de Lyon,  Unité de Mathématiques Pures et Appliquées, 69007, Lyon, France \\
		\emph{e-mail:~}\texttt{firstname.lastname@ens-lyon.fr}\\[4mm]
	\end{flushleft}

	\begin{flushleft}
		\textcolor{DeepSkyBlue4}{Camille Horbez}\\
	ENS de Lyon,  Unité de Mathématiques Pures et Appliquées, 69007, Lyon, France \\
		\emph{e-mail:~}\texttt{firstname.lastname@ens-lyon.fr}\\[4mm]
	\end{flushleft}

\noindent\textcolor{black!70}{\small The authors acknowledge support from the European Research Council, through Grant~101040507 Artin-Out-ME-OA. Views and opinions expressed are however those of the authors only and do not necessarily reflect those of the European Union or the European Research Council; neither the European Union nor the granting authority can be held responsible for them.}
\end{document}